\documentclass[10pt,reqno,a4paper]{article} 

\usepackage{anysize}
\marginsize{3.0cm}{3.0cm}{2.2cm}{2.2cm}

\usepackage[english]{babel}
\usepackage[centertags]{amsmath}
\usepackage{amsfonts,amssymb,amsthm}
\usepackage[svgnames]{xcolor}
\usepackage{comment}
\usepackage{hyperref} 
\usepackage{mathrsfs}
\usepackage[sans]{dsfont}
\usepackage{accents}

\let\OLDthebibliography\thebibliography
\renewcommand\thebibliography[1]{
  \OLDthebibliography{#1}
  \setlength{\parskip}{0pt}
  \setlength{\itemsep}{0pt plus 0.3ex} }

\numberwithin{equation}{section}

\theoremstyle{plain}
\newtheorem{theorem}{Theorem}[section]

\newtheorem{lemma}[theorem]{Lemma}

\theoremstyle{definition}
\newtheorem{remarkbase}[theorem]{Remark}
\newenvironment{remark}{\pushQED{\qed} \remarkbase}{\popQED\endremarkbase}

\newcommand{\N}{{\mathbb N}}
\newcommand{\R}{{\mathbb R}}
\newcommand{\C}{{\mathbb C}}
\newcommand{\Z}{\mathbb Z}

\newcommand{\T}{{\mathbb T}}
\renewcommand{\S}{{\mathbb S}}

\newcommand{\mD}{\mathcal{D}}
\newcommand{\mE}{\mathcal{E}}
\newcommand{\mF}{\mathcal{F}}

\newcommand{\mH}{\mathcal{H}}
\newcommand{\mI}{\mathcal{I}}

\newcommand{\mK}{\mathcal{K}}
\newcommand{\mL}{\mathcal{L}}
\newcommand{\mM}{\mathcal{M}}

\newcommand{\mR}{\mathcal{R}}
\newcommand{\mS}{\mathcal{S}}
\newcommand{\mT}{\mathcal{T}}
\newcommand{\mU}{\mathcal{U}}

\renewcommand{\a}{\alpha}
\renewcommand{\b}{\beta}
\newcommand{\g}{\gamma}

\newcommand{\e}{\varepsilon}
\newcommand{\ph}{\varphi}
\newcommand{\lm}{\lambda}
\newcommand{\Lm}{\Lambda}

\newcommand{\Om}{\Omega}
\newcommand{\om}{\omega}

\newcommand{\s}{\sigma}

\renewcommand{\th}{\vartheta}

\newcommand{\la}{\langle}
\newcommand{\ra}{\rangle}
\newcommand{\pa}{\partial}
\renewcommand{\div}{\mathrm{div}\,}
\newcommand{\grad}{\nabla}

\title{Two-dimensional capillary liquid drop: Craig-Sulem formulation on $\T^1$ and bifurcations from multiple eigenvalues of rotating waves}
 
\author{\normalsize{Giuseppe La Scala}}
\date{} 

\begin{document}

\maketitle

\noindent
\textbf{Abstract.} We consider the free boundary problem for a two-dimensional, incompressible, perfect, irrotational liquid drop of nearly circular shape with capillarity: that is, we consider the 2D version of the 3D capillary drop problem treated in Baldi-Julin-La Manna \cite{B.J.LM} and Baldi-La Manna-La Scala \cite{BLS}. In particular, we derive its Craig-Sulem formulation firstly over the circle, then over the one-dimensional flat torus; the arising equations are similar to the pure capillary Water Waves for the ocean problem, apart from conformal factors and additional terms due to curvature terms. Then,
we show its Hamiltonian structure and we derive constants of motions from symmetries, one of which is the invariance by the torus action $\mT_\alpha$. Thanks to this invariance, we show the existence of $\mT_\alpha-$orbits of rotating wave solutions (which are the analogous of travelling waves of the ocean problem) by bifurcation from multiple eigenvalues in the spirit of Moser-Weinstein \cite{Moser, Weinstein} and Craig-Nicholls \cite{Craig.Nicholls} variational approaches; in particular, we can parametrize such orbits by the angular momentum, and for each value of it they are unique. This will imply that each orbit is generated by symmetric rotating waves.

\bigskip

\emph{MSC 2020:} 35R35, 35B32, 35C07, 76B45, 35B38.

\tableofcontents

\section{Introduction}

We consider the free boundary problem for the motion of a pure capillary 2D drop, described for some time interval $t\in(0,T)$ by the equations
\begin{align}
&\div u 
 = 0 \quad \qquad\qquad\,\,\,\,\,\quad\text{in} \ \Om_t,\label{div.eq.01}
\\ 
 &\pa_t u + u \cdot \grad u + \grad p 
 = 0 \quad \text{in} \ \Om_t, 
\label{dyn.eq.01} 
\\
&
p  = \sigma_0 H_{\Om_t} \quad\qquad\qquad\quad \text{on} \ \pa \Om_t, 
\label{pressure.eq.01}
\\
&
V_t  = \la u , \nu_{\Om_t} \ra \quad\qquad\quad\,\,\,\,\, \ \text{on} \ \pa\Om_t.
\label{kin.eq.01}
\end{align}
Here, $u$ is the velocity vector field, $p$ is the pressure scalar field, $\s_0$ is the capillarity coefficient, $H_{\Om_t}$ is the curvature of $\pa\Om_t$, $V_t$ is the normal boundary velocity, $\nu_{t}$ is the outward unit normal vector field of $\pa\Om_t$, and finally $\Om_t$ is the geometrical domain of the drop which we assume to be smooth, simply connected and having boundary
\begin{equation}\label{ansatz}\pa\Om_t:=\{(1+h(t,x))x\colon\,x\in\S^1\},\qquad\S^1:=\{x\in\R^2\colon\,|x|=1\},
\end{equation}
where $h$ is the elevation function which satisfies $1+h>0$. If $|h|$ is small, we say that $\pa\Om_t$ has a \emph{nearly circular} shape.

The set of equations \eqref{div.eq.01}-\eqref{kin.eq.01} is given by the Euler equations \eqref{div.eq.01}-\eqref{dyn.eq.01} for incompressible perfect fluids with no external forces acting on the drop, the capillary boundary condition \eqref{pressure.eq.01} for the pressure and the boundary velocity condition \eqref{kin.eq.01}. The equations \eqref{div.eq.01}-\eqref{kin.eq.01} together with the ansatz \eqref{ansatz} are the 2D version of the 3D capillary drop equations treated in Baldi-Julin-La Manna \cite{B.J.LM} and Baldi-La Manna-La Scala \cite{BLS}, and in the present paper we want to get similar results with a similar approach.

The unknown fields are $u$, $p$ and $\Om_t$. If we assume $u$ to be also irrotational, then there exists a scalar function $\Phi$, which is called the \emph{velocity potential}, which solves the Dirichlet problem
\begin{equation}\label{Laplac.problem}\begin{aligned}&\Delta\Phi(t,x)=0\qquad\qquad\qquad\qquad\quad\,\,\,\,\text{in}\,\Om_t,
\\&\Phi(t,(1+h(t,x))x)=\psi(t,x)\qquad\,\,\,\,\text{on}\,\S^2.
\end{aligned}\end{equation}
By this, \eqref{kin.eq.01} is equivalent to $V_t=\la\grad\Phi,\nu_{t}\ra=:G(h)\psi$: here, we call $G(h)\psi$ the Dirichlet-Neumann operator. Moreover, since external forces on the fluid come from a potential (they are zero in this case), then the Bernoulli identity $\pa_t\Phi + \frac12|\grad\Phi|^2 + p = c(t)$ in $\Om_t$, with $c(t)$ a pure function of time. We choose $c(t):=\sigma_0$ and write $\pa_t\Phi + \frac12|\grad\Phi|^2 + p = \sigma_0$ and extend its validity up to $\pa\Om_t$ by continuity. We notice that this choice of $c(t)$ is fine; one could else consider for the velocity potential $\Phi$ the equivalence relation $\Phi_1\sim\Phi_2\iff\Phi_1-\Phi_2$ is a function of time.

One is then led to get the Craig-Sulem (equivalent) formulation of the system \eqref{div.eq.01}-\eqref{kin.eq.01}, see Theorem \ref{thm:ww.eq.}:
\begin{align}& \pa_t h = \frac{\sqrt{(1+h)^2 + |\grad_{\S^1} h|^2}}{1 + h} \, G(h)\psi,
\label{h.eq.}
\\
& \pa_t \psi =
 \frac12 \Big( G(h)\psi 
+ \frac{\la \grad_{\S^1} \psi , \grad_{\S^1} h \ra}{(1+h) \sqrt{(1+h)^2 + |\grad_{\S^1} h|^2}} \Big)^2
- \frac{|\grad_{\S^1} \psi|^2}{2(1+h)^2} 
  - \s_0 (H(h)-1),
  \label{psi.eq.}
\end{align}
In the case of 3D nearly spherical capillary drop, similar equations have been derived also in \cite{B.J.LM} and \cite{Shao.initial.notes}. However, in Theorem \ref{thm:ww.on.torus}, we will prove that the system \eqref{h.eq.}-\eqref{psi.eq.}, which is written on $\S^1$, can be also written on the torus $\T^1$ as follows:
\begin{align}&\pa_t\xi=e^{-2\xi}[G(\xi)\chi],\label{xi.eq}
\\&\pa_t\chi=e^{-2\xi}\Big[\frac12\Big(\frac{ G(\xi)\chi + \xi'\chi'}{\sqrt{1+\xi'^2}}\Big)^2 - \frac12\chi'^2 + \s_0 e^{\xi}\Big(\frac{\xi'}{\sqrt{1+\xi'^2}}\Big)'\Big] - \s_0\Big[e^{-\xi}\frac{1}{\sqrt{1+\xi'^2}} - 1\Big], \label{chi.eq}
\end{align}
where now $G$ must be meant as the Dirichlet-Neumann operator for the infinite-depth water waves. As proved in Theorem \ref{thm:Hamiltonian.structure}, these equations have a Hamiltonian structure, and the Hamiltonian is given respectively by the kinetic energy, the capillary energy and a volume term (which follows by the choice of $c(t)=\s_0$ in the Bernoulli identity):
\begin{equation}\label{Hamiltonian}\mathcal{H}(\xi,\chi):=\frac12\int_{\T^1}\chi\cdot G(\xi)\chi\,d\s + \s_0\int_{\T^1}e^\xi\sqrt{1+\xi'^2}d\s - \frac12\s_0\int_{\T^1}e^{2\xi}\,d\s.\end{equation}
From this, one gets constants of motion from symmetries; one of them is the conservation of the angular momentum
\begin{equation}\label{def.mI}\mI(\xi,\chi):=-\frac12\int_{\T^1}e^{2\xi}\chi'\,d\th=\int_{\T^1}e^{2\xi}\xi'\chi\,d\th,\end{equation}
coming from the invariance of $\mH$ with respect to the torus action $(\xi(\th),\chi(\th))\longmapsto\mT_\alpha(\xi(\th),\chi(\th))=(\xi(\th+\alpha),\chi(\th+\alpha))$, see Lemma \ref{lemma:rotation}.

The system \eqref{xi.eq}-\eqref{chi.eq} is similar to pure capillary Water Wave equations on the torus $\T^1$, with the substantial differences due to the factor $e^{-2\xi}$ and the additive term $\s_0[e^{-\xi}(\sqrt{1+\xi'^2})^{-1} + 1]$ in the equation \eqref{chi.eq}. This difference is due to the fact that to parametrize geometric tensors for $\S^1$ by using $\T^1$, we need the conformal transformation from the periodic strip $\mS:=\T^1\times\R$ to $\R^2\setminus\{0\}$ considered in \eqref{phi.map} (see Section 1.1), which introduces curvature terms when passing from the curved manifold $\S^1$ to the flat one $\T^1$, see Lemma \ref{derivative.change.strip} and also the curvature expression \eqref{mean.curvature}.

This allows us to look for solutions of the kind $(\xi(t,\th),\chi(t,\th)):=(\eta(\th+\om t),\beta(\th+\om t))$, which we will call \emph{rotating waves}. They correspond to drop space profiles which rigidly rotate in time with constant angular velocity $\om\in\R$: indeed, on $\S^1$ they read as functions of the kind $(\tilde\eta(R(\om t)x), \tilde\beta(R(\om t)x))$, where
\begin{equation*}R(\th):=\begin{pmatrix}\cos\th & -\sin\th \\ \sin\th & \cos\th
\end{pmatrix}.
\end{equation*}
Rotating waves are in particular the nonzero solutions of the equations
\begin{align}&\om\eta' 
= e^{-2\eta}G(\eta) \beta,
\label{rw.1} 
\\&
\om\beta' = e^{-2\eta}\frac12 \Big( \frac{G(\eta) \beta + \eta'\beta'}{ 
\sqrt{1+\eta'^2}} \Big)^2
- \frac12 e^{-2\eta}\beta'^2 
+ \s_0\Big(e^{-\eta}\Big[\frac{\eta'}{\sqrt{1+\eta'^2}} \Big]' - \frac{e^{-\eta}}{\sqrt{1+\eta'^2}} + 1  \Big),
\label{rw.2} 
\end{align}
which can be seen as the critical point equations for the functional 
\begin{equation}\label{lagrangian.functional}\mE:=\mH - \om(\mI-a)
\end{equation}
under the constraint $\{\mI(\eta,\beta)=a\}$, with $a\in\R$ fixed, see Lemma \ref{lemma:variational.structure.rw}.

In \cite{Moon.Wu}, it has been proved the existence of rotating waves for the Zakharov formulation of the system \eqref{div.eq.01}-\eqref{kin.eq.01} when they are required to satisfy \emph{both} the reversibility symmetry $(\eta(\th),\beta(\th))=(\eta(-\th),-\beta(-\th))$ and the $c-$fold symmetry $(\eta(\th+\frac{2\pi}{c}),\beta(\th+\frac{2\pi}{c}))=(\eta(\th),\beta(\th))$, where $c\in\N$ and $c\ge2$. Physically and geometrically speaking, in a generic $Oxy$ reference frame the reversibility symmetry corresponds to drops whose domain is invariant under the reflection about the $x-$axis, while the $c-$fold symmetry corresponds to drops whose domain is invariant under the rotation of angle $\frac{2\pi}{c}$. The existence of this kind of solutions have been proved by conformal transformation of the complex plane, starting from the Zakharov formulation of \eqref{xi.eq}-\eqref{chi.eq}. In terms of our Craig-Sulem formulation, their existence is due to the fact that they arise as bifurcations from a simple eigenvalue around $(\om_*;0,0)$, where $\om_*$ is given by \eqref{freq.c.fold} and $(\eta,\beta)=(0,0)$ is the \emph{static circle} solution: this is proved in Theorem \ref{thm:bif}, point $(iii)$, see Section 4.4.

If we throw away the reversibility symmetry assumption, in the analysis led in Section 4.2 we find out that the kernel space $V$ (see \eqref{kernel.L.}) is $3-$dimensional and not $1-$dimensional: in the former case, the reduction of $V$ to a $1-$dimensional subspace is due to the fact that reversibility assumption cancels out two generators of $V$. Then, one is prevented from using Crandall-Rabinowitz Theorem.
However, equations \eqref{rw.1}-\eqref{rw.2} have two useful properties: one is its variational structure as pointed out before, and the other one is its equivariance with the respect to the torus action, see Lemma \ref{lemma:F.invariance}, identity \eqref{F.equiv.}. Then, in the spirit of \cite{Weinstein, Moser, Craig.Nicholls}, one is able to get a \emph{unique} $\mT_\alpha$-orbit of rotating waves: this is proved in Theorem \ref{thm:main}, where $c-$fold symmetry is \emph{not} assumed, and in Theorem \ref{thm:bif}, point $(ii)$, where $c-$fold symmetry is assumed.

Finally, in Theorem \ref{thm:bif}, point $(i)$, we show the existence of rotating waves satisfying the reversibility symmetry but not the $c-$fold one. 

This implies the following consequence. When one does not make the reversibility symmetry assumption, one finds a unique family of $\mT_\alpha-$orbit of rotating waves parametrized by the angular momentum. Fixing a value of it, one can also find a rotating wave having the reversibility symmetry: by uniqueness of the $\mT_\alpha-$orbit, one discovers that it \emph{must} be spanned by the solution with such a symmetry.

\bigskip

\emph{Structure of the paper.}
In Section 2, we will firstly introduce the geometric objects we will deal with. Starting from a generic setting built on $1$-dimensional $C^2-$manifolds embedded in $\R^2$, we will specialize on the circle $\S^1$ and then on graph manifolds over it. Then we will parametrize the latter manifolds on the torus $\T^1$ and we will express over it all the geometrical objects. Finally, we will recall the Sobolev spaces $H^s$ and the space $H^{\mathfrak{s},s}$ of analytic functions with exponential decay on $\T^1$.

In Section 3, we will derive in Theorem \ref{thm:ww.eq.} the Craig-Sulem formulation \eqref{h.eq.}-\eqref{psi.eq.} of the free boundary problem \eqref{div.eq.01}-\eqref{kin.eq.01}, and then in Theorem \ref{thm:ww.on.torus} we will write its formulation \eqref{xi.eq}-\eqref{chi.eq} on the torus $\T^1$. As for the latter, in Theorem \ref{thm:Hamiltonian.structure} we will show its Hamiltonian structure and then in Lemma \ref{lemma:shift.mass} we will derive the conservation of the drop mass from the invariance with respect to the translations of $\chi$ with constants, while in Lemma \ref{lemma:rotation} we will derive the conservation of the angular momentum from the invariance with respect to the torus action. In the end, we show in Lemma \ref{lemma:reversibility} the reversible structure of \eqref{xi.eq}-\eqref{chi.eq}.

In Section 4, in Lemma \ref{ww.rotating.wave.eq} we will derive the rotating wave equations \eqref{rw.1}-\eqref{rw.2}, whose variational nature is proved in Lemma \ref{lemma:variational.structure.rw}. The Sections 4.2, 4.3 are devoted to prove the Theorem \ref{thm:main} about the existence of rotating waves with no symmetry and high finite as well as analytic regularity, while Section 4.4 we show Theorem \ref{thm:bif} which concernes the existence of rotating waves having only the reversible structure, then those having only the $c-$fold symmetry and in the end those having both symmetries, which is nothing less but what partially already proved in \cite{Moon.Wu}.

\bigskip

\emph{Related literature about the drop in 2D and 3D.} 
The problem of the fluid motion of a capillary drops
dates back to the paper \cite{Lord.Rayleigh.jets} by Lord Rayleigh, who studied both nearly-circular 2D drops for understanding capillary jets as well as nearly-spherical 3D drops, and formally discovered rotating waves. 
The formulation of the free boundary problem for the drop in the irrotational regime 
as a system of equations on $\S^1$ has been used in \cite{Beyer.Gunther} and for $\S^2$ 
in \cite{B.J.LM, Beyer.Gunther, Shao.initial.notes, Shao.G.paralinearization}.
For the latter case, its Hamiltonian structure has been obtained in \cite{Beyer.Gunther} 
and, in the present formulation, in \cite{B.J.LM}. 
The Dirichlet-Neumann operator for the 3D drops 
has been studied in \cite{Shao.G.paralinearization} (paralinearization)
and \cite{B.J.LM} (linearization, analyticity, tame estimates). 
Local well-posedness results for the Cauchy problem have been obtained in 
\cite{Beyer.Gunther, CS, Shao.G.paralinearization}, 
and continuation results and a priori estimates in \cite{JL, Shatah.Zeng}. The rigorous existence of rotating traveling waves has been proved in \cite{BLS}.
As for the 2D case, in \cite{Moon.Wu} it was proved the existence of rotating waves under reversibility and $c-$fold symmetry, 
while their numerical evidence is shown in \cite{Dyachenko}; 
we mention also \cite{MNS} for the existence of 2D steady bubbles for capillary drops and \cite{MRS} for the existence of 2D steady vortex sheets.
 
\medskip

\emph{Related literature about the bifurcation and critical point theory.}
Concerning the bifurcation technique, the present paper is inspired by \cite{Craig.Nicholls}, 
where the existence of capillary-gravity traveling waves on 2D and 3D flat torus is proved, 
relying on the proof of the Lyapunov Centre Theorem in \cite{Weinstein} and \cite{Moser}. 
We refer to \cite{Benci, Chow.Hale, Mahwin.Willem, Struwe} as far as the Benci index 
and its applications in proving multiplicity of critical points are concerned. 
We refer to \cite{Ambrosetti.Malchiodi, Ambrosetti.Prodi, Ambrosetti.Rabinowitz, 
Buffoni.Toland, Chow.Hale, Crandall.Rabinowitz, Fadell.Rabinowitz, Mahwin.Willem, Struwe}
for further literature about bifurcation and critical point theories,  
and to \cite{Berti.Bolle, Berti.Bolle.1, BLS} for some interesting applications to PDEs. 

\medskip

\emph{Related literature about traveling water waves.} 
Starting from the pioneering works \cite{Stokes}, \cite{Struik}, \cite{Levi.Civita} and \cite{Nekrasov},  
the literature about the water wave problem in the flat (i.e., non spherical) case, 
and in particular on traveling waves, is huge. 
For a recent review, we refer to \cite{HHSTWWW}.
Concerning the Hamiltonian structure and the Dirichlet-Neumann operator, 
important references are \cite{AM, CSS, Craig.Sulem, Lannes, Zakharov}; 
see also \cite{BMV, Wahlen.1} for recent developments.
For the Cauchy problem we refer, for instance, to \cite{ABZ, AD, Berti.Delort, Lannes}. 
For traveling waves we mention, e.g., 
\cite{Groves.Sun, Iooss.Plotnikov, Iooss.Plotnikov.1} for 3D problems,
\cite{ASW} for global bifurcation, 
\cite{GNPW} for Beltrami flows, 
\cite{Wahlen} for nonzero vorticity,
\cite{Jones.Toland} for secondary bifurcations.

\bigskip

\emph{Acknowledgements.} This work is supported by Italian Project PRIN 2020XB3EFL \emph{Hamiltonian and dispersive PDEs}.
The author thanks Pietro Baldi and Domenico Angelo La Manna for interesting discussions.

\section{Preliminaries} \label{sec:preliminary}

\subsection{Geometric objects on surfaces and their representation on the infinite periodic strip}

Let us consider an open set $\Om\subset\R^2$ whose boundary is at least $C^2$-regular. Let us call $\nu_\Om$ the outward unit normal vector field of $\pa\Om$: the tangent line at $x\in\pa\Om$ is
\begin{equation}\label{tangent.line}T_x(\pa\Om)=\{y\in\R^2\colon\,\langle y,\nu_\Om(x)\rangle=0\},
\end{equation}
where we are denoting by $\la\cdot,\cdot\ra$ the Euclidean scalar product.

Thus, calling $N_x(\pa\Om)$ the normal line at $x\in\pa\Om$, we have the orthogonal decomposition $\R^2=N_x(\pa\Om)\oplus T_x(\pa\Om)$, so we can define the projection operators $\Pi_{N_x(\pa\Om)}\colon\R^2\longmapsto N_x(\pa\Om)$ and $\Pi_{T_x(\pa\Om)} = I - \Pi_{N_x(\pa\Om)}$, where $I$ is the identity map; one can notice that 
\begin{equation}\label{normal.projection}\Pi_{N_x(\pa\Om)}=\nu_\Om(x)\otimes\nu_\Om(x).
\end{equation}

We can project any vector field $F\colon\R^2\longmapsto\R^2$ over $N_x(\pa\Om)$ and $T_x(\pa\Om)$. In particular, given any function $f\colon\R^2\longmapsto\R$ and its gradient $\nabla f$, we can define the normal and the tangential gradient of $f$ respectively as
\begin{align}&\nabla_\nu f(x):=\Pi_{N_x(\pa\Om)}[\nabla f(x)]=\langle\nabla f(x),\nu_\Om(x)\rangle\nu_\Om(x), \label{normal.gradient}
\\&\nabla_{\pa\Om}f(x):=\Pi_{T_x(\pa\Om)}[\nabla f(x)]=\nabla f(x) - \nabla_\nu f(x). \label{tangential.gradient}
\end{align}
Thus we can define the tangential differential as
\begin{equation}\label{tangential.differential}D_{\pa\Om}f(x):=[\nabla_{\pa\Om}f(x)]^T,
\end{equation}
a definition that can be extended to any smooth vector field $F\colon\R^2\longmapsto\R^2$ as
\begin{equation}\label{extended.tangential.differential}D_{\pa\Om}F(x):=DF(x) - DF(x)\nu_\Om(x)\otimes\nu_\Om(x),
\end{equation}
where $DF(x)$ is the differential (or Jacobian matrix) of $F$ on $\R^2$.

Given any smooth function $f\colon\,\pa\Om\longmapsto\R$ and any extension $\tilde{f}$ of it, one defines
\begin{equation}\grad_{\pa\Om}f(x):=\grad_{\pa\Om}\tilde f(x),
\end{equation}
and one can show that this definition does not depend on the choice of the extension $\tilde{f}$; same holds for vector fields $F\colon\,\pa\Om\longmapsto\R^2$.

Let us now turn our attention to the unit disk $\Om:=\mathbb{D}$, whose boundary is the unit circle $\pa\Om=\S^1$. 
Given any function $f\colon\,\S^1\longmapsto\R$, we define its $0-$homogeneous and $1-$homogeneous extensions to $\R^2\setminus\{0\}$ respectively as
\begin{equation}\label{homogeneous.extensions}\mE_0f(x):=f\Big(\frac{x}{|x|}\Big),\qquad\mE_1f(x):=|x|\mE_0 f(x),\qquad\forall x\in\R^2\setminus\{0\}
\end{equation}
and analogously for any vector field $F\colon\,\S^1\longmapsto\R^2$. By these definitions and by the fact that $\nu_{\mathbb{D}}(x)=x$ for all $x\in\S^1$, we get 
\begin{align}&\nabla_{\S^1}f(x)=\nabla(\mE_0f)(x)= (\grad\mE_0f)(x) - \la(\grad\mE_0f)(x),x\ra x, \label{tang.grad.S1}
\\&D_{\S^1}f(x)=(\grad_{\S^1}f(x))^T.
\end{align}
By definition, for all $x\in\S^1$ we have
\begin{equation}\label{grad.s1.orth.}\la\grad_{\S^1}f(x),x\ra = 0.
\end{equation}
Let us turn our attention to the deformed disk $\Om_h$:
\begin{equation}\label{Om.h}\Om_h:=\{(1+\mE_0h(x))x\,\colon\,x\in\mathbb{D}\},\qquad h\colon\,\S^1\longmapsto(-1,+\infty).
\end{equation}
Then, its boundary is
\begin{equation}\label{pa.Om.h}\pa\Om_h:=\{(1+h(x))x\,\colon\,x\in\mathbb{S}^1\},\qquad h\colon\,\S^1\longmapsto(-1,+\infty),
\end{equation}
which is diffeomorphic to $\S^1$ through the map \begin{equation}\label{gamma.diff}\g_h:\S^1 \longmapsto \pa \Om_h\qquad\g_h(x):=(1+h(x))x,\qquad x\in\S^1.\end{equation}
The 1-homogeneous extension of $\g$ on $\R^2\setminus\{0\}$ is
\begin{equation}  \label{def.tilde.h}
\mE_1 \g_h(x) :=  x \, (1 + \mE_0 h(x)) ,\qquad\forall x\in\R^2\setminus\{0\}.
\end{equation}
The Jacobian matrix, its determinant, its inverse and transpose are
\begin{align}&D (\mE_1 \g_h)  = (1 + \mE_0 h) I + x \otimes \nabla \mE_0 h, \label{Jac.tilde.gamma}
\\&\det D(\mE_1 \g_h) = (1 + \mE_0 h)^2, \label{det.D.tilde.gamma}
\\&[D (\mE_1 \g_h)]^{-1} = \frac{I}{1+ \mE_0 h}  
- \frac{x \otimes \nabla \mE_0 h}{(1 + \mE_0 h)^2}, \label{D.tilde.gamma.inv}
\\&[D (\mE_1 \g_h)]^{-T} 
= \frac{I}{1+ \mE_0 h} 
- \frac{(\grad \mE_0 h) \otimes x}{(1 + \mE_0 h)^2}. \label{D.tilde.gamma.inv.T} 
\end{align}
Then, the tangent line $T_{\g_h(x)}(\g_h(\S^1))$ to the curve $\g_h(\S^1)$ at $\g_h(x) \in \g_h(\S^1)$ can be proved to be
\begin{equation} \label{tangent.pa.Om}
T_{\g_h(x)}(\g_h(\S^1)) 
= \{ D_{\S^1} \g_h(x) v : v \in T_x(\S^1) \}
\qquad\forall x \in \S^1,
\end{equation}
while the outward unit normal vector field to $\g_h(\S^1)$ at $\g_h(x)$ is
\begin{equation}  \label{nu.h}
\nu_{h}\circ\gamma_h(x) = \frac{(1 + h(x)) x - \grad_{\S^1} h(x)}{\sqrt{(1 + h(x))^2 + |\grad_{\S^1} h(x)|^2}} = \frac{(1 + h(x)) x - \grad_{\S^1} h(x)}{J(h)}\qquad x\in\S^1,
\end{equation}
where
\begin{equation}\label{J.def}J(h):=\sqrt{(1+h)^2 + |\grad_{\S^1}h|^2}
\end{equation}

\begin{lemma}
\label{lem:formulas}
Let $\Omega_h$ be as in \eqref{Om.h},  
and let $\g_h$ be as in \eqref{gamma.diff}.
For any function $f\colon\, \pa \Omega_h \to \R$,
let $\tilde f$ be its pullback by $\g_h$, namely $\tilde f(x) = f\circ\g_h(x))$. 
Then 
\begin{equation}
\begin{split}
\nabla_{\pa \Omega_h} f\circ\gamma_h(x) &= \frac{\nabla_{\S^1} \tilde f(x)}{1+h} + \frac{\la \nabla_{\S^1}  \tilde f, \nabla_{\S^1} h \ra}{(1+h)J}\nu_{h} \circ\gamma_h(x)  \\
&=  \frac{\nabla_{\S^1}  \tilde f(x)}{1+h}  - \frac{\la \nabla_{\S^1}  \tilde f, \nabla_{\S^1} h \ra}{(1+h)J^2} \nabla_{\S^1} h + \frac{\la \nabla_{\S^1}  \tilde f, \nabla_{\S^1} h \ra}{J^2} x,
\end{split}
\label{eq.1.in.lem:formulas}
\end{equation}
\end{lemma}

\begin{proof} Given any $f\colon\, \pa\Omega_h \longmapsto \R$, we note that $f_0(x) = f(\gamma_h (|x|^{-1}x) )$ for all $x \in \R^2 \setminus \{ 0 \}$ is an extension for $f$ from $\pa\Omega_h$ to $\R^2\setminus\{0\}$
Analogously, let us consider $h_0(x) = (\mE_0 h)(x)= h(|x|^{-1}x)$. 
One has $\la \nabla f_0 , x \ra = \la \nabla h_0 , x \ra =0$ on $\R^2 \setminus \{ 0 \}$ 
and $\nabla h_0 = \nabla_{\S^1} h$ on $\S^1$. 

Taking any $\lambda>0$, it holds $(f_0 \circ \mE_1 \g_h)(\lambda x) = (f_0 \circ \mE_1 \g_h)(x)$ and thus
\[
\nabla (f_0 \circ \mE_1 \g_h)(x) = \nabla_{\S^1}(f_0 \circ \mE_1 \g_h)(x)  = \nabla_{\S^1} \tilde f(x). 
\] 
However, it holds also that
\[
\nabla (f_0 \circ \mE_1 \g_h)(x) = [D(\mE_1 \g_h)(x)]^T  \,  \nabla f_0\circ\mE_1 \g_h(x). 
\]
By \eqref{D.tilde.gamma.inv.T}
and \eqref{grad.s1.orth.}, 
for all $x \in \S^1$ we have
\[
\nabla f_0\circ\gamma_h(x) = [D(\mE_1 \g_h)(x)]^{-T}\nabla_{\S^1} \tilde f(x) 
= \frac{ \nabla_{\S^1} \tilde f(x)}{1+h(x)}. 
\]
and so, by \eqref{tangential.gradient}, we get \eqref{eq.1.in.lem:formulas}.

\end{proof}

We observe that $\R^2\setminus\{0\}$ is diffeomorphic to the periodic infinite strip $\mathcal{S}:=\R\times\T^1$ through the diffeomorphism $\phi\colon\,\mS\longmapsto\mathbb{R}^2\setminus\{0\}$, defined as
\begin{equation}\label{phi.map}\phi(\rho,\th):=e^{\rho}\begin{pmatrix}\cos\th \\ \sin\th
\end{pmatrix}\qquad\forall\th\in\T^1,\,\rho\in\R.
\end{equation}
In particular, $\phi(\T^1\times(-\infty,0))=\mathbb{D}\setminus\{0\}$ and $\phi(\{0\}\times\T^1)=\S^1$. Moreover, setting $\xi(\th):=\log(1+h\circ\phi(0,\th))=\log(1+h(\cos\th,\sin\th))$, we have
\begin{equation}\label{phi.omegah.t1graph}\g_\xi(\th):=\phi(\xi(\th),\th)=e^{\xi(\th)}\begin{pmatrix}\cos\th \\ \sin\th\end{pmatrix} = (1 + h(\cos\th,\sin\th))\begin{pmatrix}\cos\th \\ \sin\th
\end{pmatrix}.
\end{equation}
Setting
\begin{equation}\label{omegah.diffeo.to.graph}\begin{aligned}&\Omega_\xi:=\{(\rho,\th)\in\mS\colon\,\th\in\T^1,\,-\infty<\rho<\xi(\th)\},
\\&\pa\Omega_\xi = \{(\rho,\th)\in\mS\colon\,\th\in\T^1,\,\rho=\xi(\th)=\log(1+h(\cos\th,\sin\th))\},
\end{aligned}\end{equation}
we get $\phi(\Om_\xi)=\Om_h\setminus\{0\}$, $\phi(\pa\Om_\xi)=\pa\Om_h$ and that $\g_\xi\colon\,\T^1\longmapsto\pa\Om_\xi$ is a diffeomorphism from $\T^1$ to $\pa\Om_\xi$, namely a parametrization of $\pa\Om_\xi$. For all $\th\in\T^1$, denoting by $f':=\pa_\th f$ the derivative of $f$ in the $\th$ variable, one has that the derivative of $\g_\xi$ is 
\begin{equation}\label{parametrization.derivative}\g_\xi'(\th)=\begin{pmatrix}(\g_\xi')_1(\th) \\ (\g_\xi')_2(\th)\end{pmatrix}=e^{\xi(\th)}\Big\{\xi'(\th)\begin{pmatrix}\cos\th \\ \sin\th
\end{pmatrix} - \begin{pmatrix}\sin\th \\ -\cos\th
\end{pmatrix}\Big\}.
\end{equation}
If we equip $\R^2\setminus\{0\}$ with the Euclidean metric $g_{Euc}:=dx\otimes dx + dy\otimes dy$, then we can do the pullback of it through $\g_\xi$ to induce a metric $g$ on $\pa\Om_h$, whose local representation on $\T^1$ is
\begin{equation}\begin{aligned}g(\th)&=((\g_{\xi})_*g_{Euc})(\th)
\\&:=d_\th(\g_\xi)_1\otimes d_\th(\g_\xi)_2 + d_\th(\g_\xi)_2\otimes d_\th(\g_\xi)_2
\\&=[(\g_\xi')_1^2(\th)+(\g_\xi')_2^2(\th)]d\th\otimes d\th
\\&=e^{2\xi(\th)}(1+\xi'^2(\th))d\th\otimes d\th
\\&=:g_{\th\th}(\th)d\th\otimes d\th.
\end{aligned}\end{equation}
Let us also call 
\begin{equation}\label{metric.inverse}g^{\th\th}(\th):=(g_{\th\th}(\th))^{-1}=e^{-2\xi(\th)}(1+\xi'^2(\th))^{-1}.
\end{equation}
The following lemma provides a representation on $\T^1$ of $\grad_{\S^1}$, $\nu_{h}$ and some correlated quantities:

\begin{lemma}\label{derivative.change.strip}The following facts hold.

\bigskip

$(i)$ For any $f\colon\,\R^2\setminus\{0\}\longmapsto\R$, let us call $\tilde{f}:=f\circ\phi$ its representation on $\mS$. Then, 
\begin{align}&\grad_{\mS}\tilde f(\rho,\th):=\begin{pmatrix}\pa_\rho\tilde{f}(\rho,\th) \\ \pa_\th\tilde{f}(\rho,\th)\end{pmatrix} = e^\rho\begin{pmatrix}\cos\th & \sin\th \\ -\sin\th & \cos\th
\end{pmatrix}\grad f\circ\phi(\rho,\th) ,\label{grad.rho.th.to.grad.true}
\\&\grad f\circ\phi(\rho,\th) = e^{-\rho}\Big\{\pa_\rho\tilde{f}(\rho,\th)\begin{pmatrix}\cos\th \\ \sin\th
\end{pmatrix} - \pa_\th\tilde{f}(\rho,\th)\begin{pmatrix}\sin\th \\ -\cos\th\end{pmatrix}\Big\}, \label{grad.true.to.rho.th}
\\&\Delta f\circ\phi(\rho,\th)=e^{-2\rho}\Delta_{\mS}\tilde f (\rho,\th). \label{Laplace.Transformation}
\end{align}
Here, we are denoting by $\Delta_{\R^2},\,\Delta_{\mS}$ respectively the Laplace operators on $\R^2,\,\mS$.
$(ii)$ For any $f,g\colon\,\S^1\longmapsto\R$, let us call $\tilde{f}:=\mE_0f\circ\phi,\,\tilde{g}:=\mE_0g\circ\phi$ their respective representations on $\mS$. Then, $\tilde f(\th,\rho)=\tilde f(\th)$, $\tilde g(\th,\rho)=\tilde g(\th)$ for any $(\th,\rho)\in\mS$ and
\begin{align}&\grad_{\S^1}f\circ\phi(0,\th)=\tilde{f}'(\th)\begin{pmatrix}-\sin\th \\ \cos\th \label{grad.true.to.th.S1}
\end{pmatrix},
\\&\la\grad_{\S^1}f\circ\phi(0,\th),\grad_{\S^1}g\circ\phi(0,\th)\ra = \tilde{f}'(\th)\tilde{g}'(\th), \label{scalar.products.th}
\\&|\grad_{\S^1}f\circ\phi(0,\th)|^2=\tilde{f}'^2(\th),\label{norm.grad.S}
\\&\sqrt{(1+h\circ\phi(0,\th))^2 + |\grad_{\S^1}h\circ\phi(0,\th)|^2}=e^{\xi(\th)}\sqrt{1+\xi'^2(\th)}. \label{J.th.}
\end{align}
$(iii)$ Let $\nu_{h}$ be the normal vector field of $\pa\Om_h$, let $\nu_\xi:=\nu_{h}\circ\g_\xi$ be its local representations on $\mS$. Then,
\begin{equation}\label{nu.representation}\nu_\xi(\th)=(1+\xi'^2(\th))^{-\frac12}\Big\{\begin{pmatrix}\cos\th \\ \sin\th
\end{pmatrix} + \xi'(\th)\begin{pmatrix}\sin\th \\ -\cos\th
\end{pmatrix} \Big\} = |\g_\xi'(\th)|^{-1}\begin{pmatrix}0 & -1 \\ 1 & 0
\end{pmatrix} \g_\xi'(\th)
\end{equation}
$(iv)$ Let $\Phi\colon\,\Om_h\longmapsto\R$ be the solution of the problem
\begin{equation}\label{bp.on.h}\begin{aligned}&\Delta\Phi=0\qquad\qquad\,\,\text{in}\,\Om_h,
\\&\Phi\circ\g_h=\psi\qquad\quad\text{on}\,\S^1,
\end{aligned}\end{equation}
and let 
\begin{equation}\label{G.space}G(h)\psi:=\la\grad\Phi\circ\g_h,\nu_{h}\circ\g_h\ra,
\end{equation}
be the Dirichlet-Neumann operator.
Let us suppose that $\grad\Phi$ is bounded in $\Om_h$, and let $\tilde{\Phi}:=\Phi\circ\phi$ the representation of $\Phi$ on $\mS$. Then, for all $\th\in\T^1$ we have
\begin{equation}\label{G.into.G.on.T1}G(h)\psi\circ\phi(\th,0) = e^{-\xi(\th)}(1+\xi'^2(\th))^{-\frac12}\tilde{G}(\xi)\chi(\th),\end{equation}
where $\chi:=\Phi\circ\g_\xi$ and $\tilde{G}(\xi)\chi$ is the Dirichlet-Neumann operator on the torus at infinite depth, namely
\begin{equation}\tilde{G}(\xi)\chi=\sqrt{1+\xi'^2}\cdot\la\grad\tilde{\Phi}\circ\g_\xi,\nu_{\xi}\ra,
\end{equation}
where $\tilde{\Phi}\colon\,\bar{\Om}_\xi\longmapsto\R$ is the solution to the Dirichlet problem
\begin{equation}\begin{aligned}\label{generic.bp}&\Delta\tilde\Phi=0\qquad\qquad\qquad\text{in}\,\,\Om_\xi,
\\&\tilde\Phi\circ\g_\xi=\chi\qquad\qquad\,\,\,\,\text{on}\,\,\T^1,
\\&\pa_\rho\tilde\Phi(\rho,\th)\to0\qquad\quad\,\text{as}\,\,\rho\to-\infty
\end{aligned}\end{equation}  

\end{lemma}

\begin{proof} $(i)$ For all $(\th,\rho)\in\mS$, one has
\begin{align*}&\pa_\rho\tilde{f}(\rho,\th)=\pa_\rho f(e^\rho(\cos\th,\sin\th)) = e^\rho[\cos\th\cdot\pa_x f(\phi(\rho,\th)) + \sin\th\cdot\pa_y f(\phi(\rho,\th))],
\\&\pa_\th\tilde{f}(\rho,\th)=\pa_\th f(e^\rho(\cos\th,\sin\th)) = e^\rho[-\sin\th\cdot\pa_x f(\phi(\rho,\th)) + \cos\th\cdot\pa_y f(\phi(\rho,\th))],
\end{align*}
and so we straightforward get \eqref{grad.rho.th.to.grad.true}, \eqref{grad.true.to.rho.th} and also \eqref{Laplace.Transformation}.

$(ii)$ To have \eqref{grad.true.to.th.S1}, it is enough to apply \eqref{grad.true.to.rho.th} to $\mE_0 f$. The \eqref{scalar.products.th}, \eqref{norm.grad.S} and \eqref{J.th.} are given by straightforward computations.

$(iii)$ By \eqref{nu.h}, \eqref{grad.true.to.th.S1}, \eqref{scalar.products.th}, \eqref{norm.grad.S} and \eqref{J.th.} we have
\begin{align*}\nu_\xi(\th)&=\frac{(1 + h(\cos\th,\sin\th))\begin{pmatrix}\cos\th \\ \sin\th
\end{pmatrix} - (\grad_{\S^1} h)(\cos\th,\sin\th)}{\sqrt{(1 + h(\cos\th,\sin\th))^2 + |(\grad_{\S^1} h)(\cos\th,\sin\th)|^2}}
\\&=\frac{1}{e^{\xi(\th)}\sqrt{1 + \xi'^2(\th)}}\Big\{e^{\xi(\th)}\begin{pmatrix}\cos\th \\ \sin\th
\end{pmatrix} - e^{\xi(\th)}\xi'(\th)\begin{pmatrix}-\sin\th \\ \cos\th  \end{pmatrix}\Big\}
\\&=\frac{1}{\sqrt{1+\xi'^2(\th)}}\Big\{\begin{pmatrix}\cos\th \\ \sin\th\end{pmatrix} + \xi'(\th)\begin{pmatrix}\sin\th \\ -\cos\th
\end{pmatrix} \Big\},
\end{align*}
which gives the first identity. The second identity can be directly checked by \eqref{parametrization.derivative}.

$(iv)$  Let $\Phi\colon\,\bar{\Om}_h\longmapsto\R$ be the solution of \eqref{generic.bp} with boundary datum $\psi$, and let us call $\tilde{\Phi}:=\Phi\circ\phi$ its representation on $\bar{\Om}_\xi=\phi^{-1}(\bar{\Om}_h)$: then, we get that $\tilde\Phi$ solves the problem \eqref{generic.bp}, since the first equation comes from \eqref{Laplace.Transformation} the boundary condition from the definition of $\chi$ and the infinity condition from the fact that $\grad\Phi$ is bounded and so by \eqref{grad.rho.th.to.grad.true}, \begin{equation*}|\pa_\rho\tilde\Phi(\rho,\th)|=e^{\rho}|\la(\cos\th,\sin\th),\grad\Phi\circ\phi(\rho,\th)\ra|\to0\qquad\text{as}\quad \rho\to-\infty.\end{equation*}
Also, by \eqref{grad.true.to.rho.th} and \eqref{nu.representation} we have
\begin{align*}G(h)\psi\circ\phi(0,\th)&=\la\grad\Phi\circ\g_\xi(\th), \nu_{h}\circ\g_\xi(\th)\ra
\\&=e^{-\xi(\th)}(1+\xi'^2(\th))^{-\frac12}[\pa_\rho\tilde{\Phi}(\xi(\th),\th) - \pa_\th\tilde{\Phi}(\xi(\th),\th)\xi'(\th)]
\\&=e^{-\xi(\th)}(1+\xi'^2(\th))^{-\frac12}\tilde{G}(\xi)\chi(\th).
\end{align*}

\end{proof}

\begin{remark}\label{rem} We notice that if $\tilde\Phi$ is a (unique) smooth solution of \eqref{generic.bp}, then $\tilde\Phi=O(e^{\rho})$ as $\rho\to-\infty$. To see this, let us first suppose that $\xi\equiv0$: looking for $\tilde{\Phi}$ of the kind
\begin{equation*}\tilde{\Phi}(\rho,\th)=\sum_{n\in\Z}\Phi_n(\rho)e^{in\th},\qquad\tilde{\Phi}(0,\th)=\chi(\th)=\sum_{n\in\Z}\chi_ne^{in\th},\qquad\chi_{-n}=\bar{\chi}_n
\end{equation*}
then one gets
\begin{align*}&\pa_{\rho\rho}\Phi_n=n^2\Phi_n,
\\&\Phi_n(0)=\chi_n(0),\qquad\lim_{\rho\to-\infty}\pa_\rho\Phi_n(\rho)=0,
\end{align*}
which is uniquely solved by $\Phi_n(\rho)=\chi_n e^{|n|\rho}$. As a result, one has
\begin{align*}&\Phi(\rho,\th)=\sum_{n\in\Z}\psi_n e^{|n|\rho}e^{in\th},
\\&\grad\Phi(\rho,\th)=\sum_{n\in\Z}\begin{pmatrix}|n| \\ in
\end{pmatrix}\psi_n e^{|n|\rho}e^{in\th},
\end{align*}
and so one discovers that there exists the limit
\begin{equation*}\lim_{\rho\to-\infty}e^{-\rho}\grad\Phi(\rho,\th)=\begin{pmatrix}1 \\ i
\end{pmatrix}\chi_1 e^{i\th} + \begin{pmatrix}1 \\ -i
\end{pmatrix}\bar{\chi}_1 e^{-i\th}.
\end{equation*}
Similar computations hold even if $\xi\not\equiv0$, because the solution of \eqref{generic.bp} is an extension of the solution of the problem
\begin{align*}&\Delta\tilde\Phi=0\qquad\qquad\qquad\,\,\,\,\text{in}\,\,\Om_\xi,
\\&\tilde\Phi=\tilde{\Phi}|_{\{\log\|\xi\|_\infty\}\times\T^1}\quad\,\,\text{on}\,\,\{\log\|\xi\|_\infty\}\times\T^1,
\\&\pa_\rho\tilde\Phi(\rho,\th)\to0\qquad\quad\,\,\,\,\,\text{as}\,\,\rho\to-\infty.    
\end{align*}
As a consequence, by similar computations done in the proof of Lemma \ref{derivative.change.strip}, point $(iv)$, we get that $\Phi:=\tilde\Phi\circ\phi^{-1}$ is a (unique) smooth solution of \eqref{bp.on.h}.
\end{remark}

We use the metric $g_h$ and the representation \eqref{nu.representation} to define the curvature of $\pa\Om_h$:
\begin{equation}\label{curvature}(H(h))(\th):=-g^{\th\th}(\th)\la \g_\xi''(\th), \nu_\xi(\th)\ra.
\end{equation}
Geometrically speaking, it is nothing less but the opposite of the trace of the Second Fundamental Form $II(\th):=\la\g_\xi''(\th), \nu_\xi(\th)\ra d\th\otimes d\th$ of $\pa\Om_h$ through the metric $g$, see for instance \cite{GHL}. By a direct computation, one can check that
\begin{equation}\begin{aligned}\label{mean.curvature}H(h)&=-(1+\xi'^2))^{-\frac32}e^{-\xi}\{\xi'' - 1-\xi'^2\}
\\&=e^{-\xi}\Big[ \frac{1}{\sqrt{1+\xi'^2}} - \Big(\frac{\xi'}{\sqrt{1+\xi'^2}}\Big)' \Big].
\end{aligned}\end{equation}
One can notice that in the last identity, the first term is the curvature of the graph of $\xi(\th)$ up to the conformal factor $e^{-\xi(\th)}$. Let us notice that $H(0)=1$, which corresponds to the curvature of the unit circle.

\subsection{The spaces of Sobolev and analytic functions on $\T^1$}

We introduce here the functional setting we will work with. 
All the functions $f\in L^2$ can be expanded in Fourier series along the $L^2-$orthonormal basis $\{\ph_{\ell,m}\colon\,(\ell,m)\in\mT\}$, with $\mT:=\{(0,0)\}\cup(\N\times\{-1,1\})$:
\begin{equation*}f=\sum_{(\ell,m)\in\mT}f_{\ell,m}\ph_{\ell,m};
\end{equation*}
here,
\begin{equation}\label{Fourier.basis}\ph_{\ell,m}(\th):=\begin{cases}C_{0} & \text{if}\qquad \ell=m=0,
\\C_\ell\cos(\ell\th)& \text{if}\qquad \ell>0,m=1,
\\C_\ell\sin(\ell\th)& \text{if}\qquad \ell>0,m=-1,
\end{cases}
\end{equation}
One defined the Sobolev space $H^s$ as
\begin{equation}H^s=\Big\{f\in L^2\colon\,\|f\|_{H^s}^2=\sum_{(\ell,m)\in\mT}(1+|\ell|^{2s})|f_{\ell,m}|^2<+\infty\Big\},
\end{equation}
In what will follow, we will denote by $(H^s)^2:=H^s\times H^s$.

We define also the class of analytic functions with exponential decay rate $\mathfrak{s}\ge0$ on torus:
\begin{equation}H^{\mathfrak{s},s}:=\Big\{f\in L^2\colon\,\|f\|_{H^{\mathfrak{s},s}}^2:=\sum_{(\ell,m)\in\mT}e^{2\mathfrak{s}\ell}(1+|\ell|^{2s})|f_{\ell,m}|^2<+\infty \Big\}.
\end{equation}

The following lemma concerns some algebraic and regularity properties of the Dirichlet-Neumann operator for Water Wave equations in infinite depth. 

\begin{lemma}[from \cite{Lannes, Lannes1, BMV}]\label{lemma:G}
Calling $\dot{H}$ the space of Sobolev functions with zero average, one has the following facts.

$(i)$(Low norm estimate) For all fixed $\xi\in H^{s_0}$ with $s_0>\frac32$, one has that $\tilde G(\xi)[\cdot]$ is a bounded operator from $H^{\frac12}$ to $H^{-\frac12}$, that is, 
\begin{equation}\label{G.low.norm.estimate}\|\tilde G(\xi)\chi\|_{H^{-\frac12}}\le C(\xi)\cdot\|\chi\|_{H^\frac12},
\end{equation}
where $C(\xi)$ is a constant depending on $\xi$.

$(ii)$(Symmetricity and nonnegativity) The operator $\chi\in H^{\frac12}\longmapsto\tilde G(\xi)\chi\in H^{-\frac12}$ is $L^2-$symmetric and nonnegative for all fixed $\xi\in W^{1,\infty}$, that is, for all $\chi_1,\chi_2\in H^{\frac12}$,
\begin{align}&\la \tilde G(\xi)\chi_1,\chi_2\ra_{L^2} = \la \tilde G(\xi)\chi_2,\chi_1\ra_{L^2},
\label{G.is.symmetric}
\\&\la\tilde G(\xi)\chi_1,\chi_1\ra_{L^2}\ge0.\label{G.is.nonnegative}
\end{align}
$(iii)$(Constants are in $\ker\tilde G(\xi)[\cdot]$) For all constant functions $\chi\equiv c\in\R$, for all fixed $\xi\in H^{s_0}$ with $s_0>\frac32$,
\begin{equation}\tilde G(\xi)c=0. \label{G.is.zero.on.constants}
\end{equation}
$(iv)$(Invariance by inversion of sign) Let $\xi\in H^{s_0}$ with $s_0>\frac32$, $\chi\in H^{\frac12}$, and let us call $\iota$ the operator of inversion of sign, that is, for all functions $f\colon\,\T^1\longmapsto\R$ and $\th\in\T^1$, $(\iota f)(\th):=f(-\th)$. Then,
\begin{equation}\label{G.inversion}\tilde G(\iota\xi)[\iota\chi]=\iota(\tilde G(\xi)\chi).
\end{equation}
$(v)$(Translation invariance of $\tilde G$) Let $\mT_\alpha$ be the translation operator, that is, $(\mT_\alpha f)(\th):=f(\th+\alpha)$ for any $f\colon\,\T^1\longmapsto\R$, $\th\in\T^1$ and $\alpha\in\T^1$. Then, for any $\alpha\in\T^1$, $\xi\in H^{s_0}$ with $s_0>\frac32$ and $\chi\in H^{\frac12}$,
\begin{equation}\label{G.translation.inv}\mT_\alpha(\tilde G(\xi)\chi)=\tilde G(\mT_{\alpha}\xi)[\mT_{\alpha}\chi].
\end{equation}
$(vi)$(High norm estimates) Let $s_0>\frac12$, let $s\ge0$ and $\xi\in H^{s+\frac12}\cap H^{s_0+1}$. Then, 
\begin{align}&\|\tilde G(\xi)\chi\|_{H^{s-\frac12}}\le C(\|\xi\|_{s_0+1},\|\xi\|_{s+\frac12})\cdot\|\chi'\|_{H^{s-\frac12}},\qquad\forall0\le s\le s_0+\frac32,\,\forall\chi\in\dot{H}^{s-\frac12}, \label{G.High.norm.est.low}
\\& \|\tilde G(\xi)\chi\|_{H^{s-\frac12}}\le C(\|\xi\|_{s_0+2})[\|\chi'\|_{H^{s-\frac12}} + \|\xi\|_{H^{s+\frac12}}\|\chi'\|_{H^{s_0+1}}],\qquad\forall0\le s\le s_0+\frac32,\,\forall\chi\in\dot{H}^{s-\frac12}\label{G.High.norm.est.high}
\end{align}

$(vii)$(Analiticity in finite regularity) $s_0>\frac12$, $0\le s\le s_0+\frac12$ and $\|\xi\|_{H^{s_0 + 1}}<\delta_0$. Then, the map $\xi\in H^{s_0+1}\longmapsto\tilde G(\xi)\in\mathcal{L}(H^{s+\frac12},H^{s-\frac12})$ is analytic.

$(viii)$(Shape derivative) Let, $s_0>\frac12$, $\chi\in H^{\frac32}$, $\xi\in H^{s_0+1}$. Then, for all directions $\hat\xi\in H^{s_0+1}$ corresponding to the $\xi-$variable,
\begin{equation}\label{shape.derivative}d(\tilde G(\xi)\hat\xi)\chi=-\tilde G(\xi)[B\hat\xi] - (V\hat\xi)', \end{equation}
where
\begin{equation}\label{B.V.}B:=\frac{\tilde G(\xi)\chi + \xi'\chi'}{1+\xi'^2},\qquad V:=\chi'-B\xi'.
\end{equation}

$(ix)$(Analiticity for analytic regularity) Let $\mathfrak{s}>0$, let $s,s_0$ such that $s,s_0+\frac12\in\N$ and $s-\frac32\ge s_0>1$. Then, there exists $\e_0:=\e_=(s)>0$ such that for all $\xi$ for which $\|\xi\|_{H^{\mathfrak{s},s_0+\frac32}}<\e_0$, the map $\xi\in H^{\mathfrak{s},s}\longmapsto\tilde G(\xi)\in\mathcal{L}(H^{\mathfrak{s},s},H^{\mathfrak{s},s-1})$ is analytic and in particular satisfies the tame estimate
\begin{equation}\label{analiticity.for.analytic.functions}\|\tilde G(\xi)\chi\|_{H^{\mathfrak{s},s-1}}\le C(s)[\|\chi\|_{H^{\mathfrak{s},s}} + \|\xi\|_{H^{\mathfrak{s},s}}\cdot\|\chi\|_{H^{\mathfrak{s},s_0+\frac32}}].
\end{equation}

\end{lemma}

\section{The capillary drop equations}

\subsection{The Craig-Sulem formulation on $\S^1$ and $\T^1$}

In this section, we will derive the $2D$ capillary drop equations, we will show their equivalence with the free-boundary equations \eqref{div.eq.01}-\eqref{kin.eq.01} and finally we rewrite them on the torus $\T^1$.

Let us recall the equations \eqref{div.eq.01}-\eqref{kin.eq.01}:
\begin{align*}
&\div u 
 = 0 \quad \qquad\qquad\,\,\,\,\,\quad\text{in} \ \Om_t,
\\ 
 &\pa_t u + u \cdot \grad u + \grad p 
 = 0 \quad \text{in} \ \Om_t, 
\\
&
p  = \sigma_0 H_{\Om_t} \quad\qquad\qquad\quad \text{on} \ \pa \Om_t, 
\\
&
V_t  = \la u , \nu_{t} \ra \quad\qquad\quad\,\,\,\,\, \ \text{on} \ \pa \Om_t.
\end{align*}
We recall that $\Om_t$ is a star-shaped domain whose boundary is of the form
\begin{equation}\label{pa.Om.t}\pa\Om_t:=\{(1+h(t,x))x\,\colon\,x\in\S^1\},\qquad h(t,\cdot)\colon\,\S^1\longmapsto(-1,+\infty),\qquad t\in\R
\end{equation}
and so the map $\g_t\colon\,\S^1\longmapsto\pa\Om_t$, which is given by 
\begin{equation}\label{Om.t.par.}\g_t(x):=(1+h(t,x))x,\end{equation}
for all $t\in\R,\,x\in\S^1$ (see \eqref{gamma.diff}), is a diffeomorphism between $\S^1$ and $\pa\Om_t$. Define the normal boundary velocity as
\begin{equation}\label{boundary.velocity.Om.t} V_t(\g_t(x)):=\la\pa_t\g_t(x), \nu_{t}\ra,
\end{equation}
where $\nu_t$ is the outer normal vector field of $\pa\Om_t$.

Let us recall that, by assumption, $u$ is curl-free, that is, $\text{curl}\,u=0$. Since $\pa\Om_t$ is star-shaped, there exists a scalar potential $\Phi$ in $\Om_t$, that is, \begin{equation}\label{potential.u}u=\grad\Phi\quad\text{in}\,\Om_t;\end{equation}
we recall also that $\Phi$ solves the Dirichlet problem \eqref{Laplac.problem}, which is
\begin{equation*}\begin{aligned}&\Delta\Phi(t,\cdot)=0\qquad\qquad\qquad\qquad\text{in}\,\Om_t,
\\&\Phi(t,\cdot)\circ\g_t=\psi(t,\cdot)\qquad\qquad\quad\text{on}\,\S^2.
\end{aligned}\end{equation*}

Let us derive the Craig-Sulem formulation of \eqref{div.eq.01}-\eqref{kin.eq.01}, and show that they are equivalent:
\begin{theorem}\label{thm:ww.eq.}In the notations above, the equations \eqref{div.eq.01}-\eqref{kin.eq.01} are equivalent to the Water Wave equations \eqref{h.eq.}-\eqref{psi.eq.}, which we remember to be
\begin{align*}& \pa_t h = \frac{\sqrt{(1+h)^2 + |\grad_{\S^1} h|^2}}{1 + h} \, G(h)\psi,
\\
& \pa_t \psi =
 \frac12 \Big( G(h)\psi 
+ \frac{\la \grad_{\S^1} \psi , \grad_{\S^1} h \ra}{(1+h) \sqrt{(1+h)^2 + |\grad_{\S^1} h|^2}} \Big)^2
- \frac{|\grad_{\S^1} \psi|^2}{2(1+h)^2} 
  - \s_0 (H(h)-1),
\end{align*}
where $h,\psi\colon\,\S^1\longmapsto\R$.

\end{theorem}

\begin{proof}
Let $\Om_t\subset\R^2$ and $u\colon\,\Om_t\longmapsto\R^2$ the solution to \eqref{div.eq.01}-\eqref{kin.eq.01} hold. Let us start by getting \eqref{h.eq.}.
As in \eqref{nu.h}, the explicit formula for $\nu_{t}$ is
\begin{equation}\label{normal.Om.t}\nu_{t}\circ\g_t(x)=\frac{(1 + h(t,x)) x - \grad_{\S^1} h(t,x)}{J(h)},\qquad x\in\S^1.\end{equation}
Then, on one hand by \eqref{Om.t.par.}, \eqref{boundary.velocity.Om.t} and \eqref{normal.Om.t} one gets
\begin{equation}\label{boundary.velocity}V_t\circ\g_t=\frac{1+h}{J(h)}\pa_t h,
\end{equation}
and on the other hand, by \eqref{kin.eq.01} it holds that
\begin{equation}\label{def.boundary.vel.}\begin{aligned}V_t\circ\g_t&=\la u(t,\cdot)\circ\g_t,\nu_{t}\circ\g_t \ra = G(h)\psi;
\end{aligned}
\end{equation}
putting together \eqref{boundary.velocity} and \eqref{def.boundary.vel.}, we obtain \eqref{h.eq.}.

Now, we want to get \eqref{psi.eq.}.
By using \eqref{potential.u} we observe that from the equation \eqref{dyn.eq.01} we get the Bernoulli identity \eqref{Ber.law}:
\begin{equation*}\pa_t\Phi + \frac12|\grad\Phi|^2 + p = c(t) \qquad\text{in}\,\Om_t.
\end{equation*}
By continuity this equation still holds on the boundary $\pa\Om_t$, that is,
\begin{equation}\label{Ber.law}\pa_t\Phi + \frac12|\grad\Phi|^2 + \s_0 H_{\Om_t} = c(t) \qquad\text{on}\,\pa\Om_t.
\end{equation}
By choosing $c(t):=\s_0$ we get
\begin{equation}\label{Bernoulli}\pa_t\Phi(t,\cdot)\circ\g_t + \frac12|\grad\Phi(t,\cdot)\circ\g_t|^2 + \s_0 H_{\Om_t} = \s_0 \qquad\text{on}\,\pa\Om_t,
\end{equation}
For all $t\in\R,\,x\in\S^1$ one has
\begin{align}&\pa_t\psi=\pa_t(\Phi(t,\cdot)\circ\g_t)=\pa_t\Phi(t,\cdot)\circ\g_t + \langle\grad\Phi(t,\cdot)\circ\g_t,\pa_t\g_t\rangle, \label{pa.t.Phi}
\\&|\grad\Phi(t,\cdot)\circ\g_t|^2 = |\grad_{\pa\Om_t}\Phi(t,\cdot)\circ\g_t|^2 + |\grad_{\nu_{t}}\Phi(t,\cdot)\circ\g_t|^2. \label{grad.Phi}
\end{align}
By using \eqref{Om.t.par.}, \eqref{h.eq.}, \eqref{normal.gradient}, \eqref{tangential.gradient} and \eqref{eq.1.in.lem:formulas} we observe that
\begin{equation}\label{grad.Phi.gamma.t}\begin{aligned}\langle\grad\Phi(t,\cdot)\circ\g_t,\pa_t\g_t\rangle&=\frac{J(h)}{1+h}G(h)\psi\cdot\{\la\nabla_{\pa\Omega_t}\Phi(t,\cdot)\circ\g_t,x\ra + \la\nabla_{\nu_{t}}\Phi(t,\cdot)\circ\g_t,x\ra\}
\\&=\frac{J(h)}{1+h}G(h)\psi\cdot\Big\{\frac{\la\grad_{\S^1}\psi,\grad_{\S^1}h\ra}{J(h)} + G(h)\psi\cdot\la\nu_{t}\circ\g_t,x\ra\Big\}
\\&=\frac{J(h)}{1+h}G(h)\psi\cdot\Big\{\frac{\la\grad_{\S^1}\psi,\grad_{\S^1}h\ra}{J(h)} + G(h)\psi\cdot\frac{1+h}{J(h)}\Big\}
\\&=\frac{\la\grad_{\S^1}\psi,\grad_{\S^1}h\ra\cdot G(h)\psi}{1+h} + (G(h)\psi)^2
\end{aligned}\end{equation}
in \eqref{grad.Phi}, one gets $\grad_{\nu_{t}}\Phi(t,\cdot)\circ\g_t=(G(h)\psi)\nu_{t}\circ\g_t$ and so 
\begin{equation}\label{norm.grad.nu.Phi}|\grad_{\nu_{t}}\Phi(t,\cdot)\circ\g_t|^2=(G(h)\psi)^2
\end{equation}
while by \eqref{eq.1.in.lem:formulas}, we have
\begin{align}\grad_{\pa\Om_t}\Phi(t,\cdot)\circ\g_t&=\frac{\nabla_{\S^1} \psi}{1+h} + \frac{\la\nabla_{\S^1}\psi, \nabla_{\S^1} h \ra}{(1+h)J(h)}\nu_{t}\circ\gamma_t. \label{grad.tang.Phi}
\end{align}
Therefore,
\begin{equation}\label{norm.grad.tan.Phi}|\grad_{\pa\Om_t}\Phi(t,\cdot)\circ\g_t|^2=\frac{|\grad_{\S^1}\psi|^2}{(1+h)^2} - \frac{|\la\nabla_{\S^1}\psi, \nabla_{\S^1} h \ra|^2}{(1+h)^2\cdot J^2(h)}.
\end{equation}
Using \eqref{pa.t.Phi}, \eqref{grad.Phi}, \eqref{grad.Phi.gamma.t}, \eqref{norm.grad.nu.Phi}, \eqref{grad.tang.Phi} and \eqref{norm.grad.tan.Phi}, \eqref{Bernoulli} becomes
\begin{equation}\begin{aligned}\label{psi.eq.1}\pa_t \psi - &\frac{\la\grad_{\S^1}\psi,\grad_{\S^1}h\ra\cdot G(h)\psi}{1+h} - (G(h)\psi)^2
+\frac12(G(h)\psi)^2 
\\&+ \frac{|\grad_{\S^1}\psi|^2}{2(1+h)^2} - \frac12\cdot\frac{|\la\nabla_{\S^1}\psi, \nabla_{\S^1} h \ra|^2}{(1+h)^2\cdot J^2(h)} + \sigma_0(H(h)-1)=0,
\end{aligned}\end{equation}
from which we get \eqref{psi.eq.}.

Conversely, let us assume that \eqref{h.eq.}, \eqref{psi.eq.} hold. 
with $\mathrm{curl}\, u = 0$. 
Define the set $\Om_t$ be with boundary $\pa\Om_t$ as in \eqref{pa.Om.t}, and $\Phi(t,\cdot)$ in $\Om_t$ 
as the solution of the Laplace problem \eqref{Laplac.problem}.
Then $\Phi(t,\cdot)$ satisfies the incompressibility condition $\Delta \Phi=0$ in $\Om_t$. 
By \eqref{h.eq.} and \eqref{Laplac.problem}, 
equation \eqref{kin.eq.01} is also satisfied. 
From \eqref{psi.eq.}, using \eqref{h.eq.}, we obtain  \eqref{Bernoulli}. 
Now we define $p$ on the closure of $\Om_t$ as 
\begin{equation} \label{def.tilde.p.using.dyn.eq.02}
p := - \pa_t \Phi - \frac12 |\grad \Phi|^2  
\quad \ \text{in } \overline{\Om_t} = \Om_t \cup \pa \Om_t.
\end{equation}
Then the dynamics equation \eqref{Ber.law} in the open domain $\Om_t$ trivially holds. 
From \eqref{def.tilde.p.using.dyn.eq.02} at the boundary $\pa \Om_t$ 
and \eqref{Bernoulli} (which is an identity for points of the boundary $\pa \Om_t$)
we deduce that $p = \s_0 H_{\Om_t}$ on $\pa \Om_t$, that is \eqref{pressure.eq.01}.

\end{proof}

Next, we write the capillary drop equations on the torus $\T^1$:
\begin{theorem} \label{thm:ww.on.torus} In the notations above, the equations \eqref{h.eq.}-\eqref{psi.eq.} can be written on $(t,\th)\in\R\times\T^1$ as \eqref{xi.eq}-\eqref{chi.eq}, that is,
\begin{align*}&\pa_t\xi=e^{-2\xi}[\tilde{G}(\xi)\chi],
\\&\pa_t\chi=e^{-2\xi}\Big[\frac12\Big(\frac{\tilde G(\xi)\chi + \xi'\chi'}{\sqrt{1+\xi'^2}}\Big)^2 - \frac12\chi'^2 + \s_0 e^{\xi}\cdot\Big(\frac{\xi'}{\sqrt{1+\xi'^2}}\Big)'\Big] - \s_0\Big[e^{-\xi}\frac{1}{\sqrt{1+\xi'^2}} - 1\Big].
\end{align*}

\end{theorem}

\begin{proof} 
Let us consider the equations
\begin{align*}& \pa_t h = \frac{\sqrt{(1+h)^2 + |\grad_{\S^1} h|^2}}{1 + h} \, G(h)\psi,
\\
& \pa_t \psi =
 \frac12 \Big( G(h)\psi 
+ \frac{\la \grad_{\S^1} \psi , \grad_{\S^1} h \ra}{(1+h) \sqrt{(1+h)^2 + |\grad_{\S^1} h|^2}} \Big)^2
- \frac{|\grad_{\S^1} \psi|^2}{2(1+h)^2} 
  - \s_0 (H(h)-1),
\end{align*}
Let us transform the equation for $\pa_t h$.
Since for all $x=\phi(0,\th)\in\S^1$ for $\th\in\T^1$, we have $1+h(t,x)=1+h(t,\phi(0,\th))=e^{\xi(t,\th)}$, hence
\begin{equation}\label{pa.t.h.strip}\pa_t h(t,x)=e^{\xi(t,\th)}\pa_t\xi(t,\th).
\end{equation}
Then, by \eqref{pa.t.h.strip}, Lemma \ref{derivative.change.strip}, point $(iii)$, identity \eqref{J.th.}, and Lemma \ref{derivative.change.strip} point $(iv)$, we have
\begin{equation*}e^\xi\pa_t\xi=e^{-\xi}\tilde{G}(\xi)\chi,
\end{equation*}
from which we get \eqref{xi.eq}.

Let us now transform the equation for $\pa_t\psi$.
Since for all $x=\phi(\th,0)\in\S^1$ with $\th\in\T^1$, we have $\psi(t,x)=\psi(t,\cdot)\circ\phi(0,\th)=\chi(t,\th)$ and so by Lemma \ref{derivative.change.strip} point $(iv)$, Lemma \ref{derivative.change.strip}, point $(iii)$, equations \eqref{scalar.products.th}, \eqref{norm.grad.S} and \eqref{J.th.}, and by \eqref{mean.curvature} we have
\begin{align*}\pa_t\chi(t,\th)&=\pa_t\psi(t,\cdot)\circ\phi(0,\th)
\\&=\frac12\Big(e^{-\xi}(1+\xi'^2)^{-\frac12}\tilde G(\xi)\chi + e^{-\xi}(1+\xi'^2)^{-\frac12}\xi'\chi'\Big)^2-e^{-2\xi}\cdot\frac12\chi'^2 
\\&\qquad- \s_0 e^{-\xi}\cdot\Big( \Big(\xi'(1+\xi'^2)^{-\frac12}\Big)' - (\sqrt{1+\xi'^2})^{-\frac12}-e^\xi \Big),
\end{align*}
which is exactly \eqref{chi.eq}.

\end{proof}

\begin{remark} Thanks to Remark \ref{rem}, one can observe that for $\psi$ smooth enough (true if and only if $\chi$ is smooth enough), then equation \eqref{h.eq.}-\eqref{psi.eq.} and \eqref{xi.eq}-\eqref{chi.eq} are equivalent.
\end{remark}

From now on, we consider the notation 
\begin{equation*}G(\xi)\chi:=\tilde G(\xi)\chi.\end{equation*}

\subsection{Hamiltonian structure, symmetries and invariants of motion}

In this section, we prove that the capillary drop equations \eqref{h.eq.}-\eqref{psi.eq.} have a Hamiltonian structure, and then, in the spirit of \cite{BLS}, we employ it to find some conserved quantities. 

A candidate for the Hamiltonian function is the total energy of the drop:
\begin{equation}\mH:=\underbrace{\frac12\int_{\Om_h}|u|^2\,dx}_{=:(I)} + \underbrace{\s_0\int_{\pa\Om_h}1\,d\mH^1}_{=:(II)} - \underbrace{\s_0\int_{\Om_h}1\,dx}_{=:(III)}.
\end{equation}
Let us compute $(I)$ by applying \eqref{potential.u}, \eqref{Laplac.problem}, Divergence Theorem and \eqref{G.into.G.on.T1}:
\begin{align*}(I)
&=\frac12\int_{\Om_h}|\grad\Phi|^2\,dx
\\&=\frac12\int_{\pa\Om_h}\Phi\la\grad\Phi,\nu_h\ra\,d\mH^1
\\&=\frac12\int_{\pa\Om_h}\Phi\,G(h)\psi\,d\mH^1
\\&=\frac12\int_{\S^1}\psi\,G(h)\psi\,J\,dx
\\&=\frac12\int_{\T^1}\chi\,G(\xi)\chi\,d\th
\end{align*}
Let us compute $(II)$:
\begin{align*}(II)
&=\s_0\int_{\S^1}J\,dx = \s_0\int_{\T^1}e^\xi\sqrt{1 + \xi'^2}\,d\th.
\end{align*}
Finally, let us compute $(III)$ by using Divergence Theorem, \eqref{grad.s1.orth.}:
\begin{align*}(III)
&=\frac{\s_0}{2}\int_{\pa\Om_h}\la x,\nu_h\ra\,d\mH^1
\\&=\frac{\s_0}{2}\int_{\S^1}\la(1+h)x,(1+h)x - \grad_{\S^1}h\ra\,dx
\\&=\frac{\s_0}{2}\int_{\S^1}(1+h)^2\,dx
\\&=\frac{\s_0}{2}\int_{\T^1}e^{2\xi}\,d\th.
\end{align*}
Putting all together, we get
\begin{equation*}\mH=\mH(\xi,\chi)=\frac12\int_{\T^1}\chi\, G(\xi)\chi\,d\th + \s_0\int_{\T^1}e^\xi\sqrt{1+\xi'^2}\,d\th - \frac{\s_0}{2}\int_{\T^1}e^{2\xi}\,d\th
\end{equation*}

The Hamiltonian structure of \eqref{h.eq.}-\eqref{psi.eq.} is shown in the following lemma.

\begin{lemma} \label{thm:Hamiltonian.structure}
The following facts hold.

$(i)$ Over any functional $A(\xi,\chi),\,B(\xi,\chi)$ smooth enough, the bilinear map defined by
\begin{equation}\begin{aligned}\label{Poisson.brackets}\{A,B\}(\xi,\chi):&=\la e^{-2\xi}\pa_\xi A(\xi,\chi),\pa_\chi B(\xi,\chi)\ra_{L^2} - \la e^{-2\xi}\pa_\chi A(\xi,\chi),\pa_\xi B(\xi,\chi)\ra_{L^2}
\\&=\la e^{-2\xi}J_0\grad A(\xi,\chi),\grad B(\xi,\chi)\ra_{L^2}
\end{aligned}\end{equation}
are Poisson brackets; here, we are denoting by $\pa_\xi A,\pa_\chi A$ respectively the $L^2-$gradients of $A$ with respect to $\xi$ and $\chi$, and by $\grad A=(\pa_\xi A,\pa_\chi A)$, and the same holds for $B$, and 
\begin{equation}\label{J0}J_0:=\begin{pmatrix}0 & -1 \\ 1 & 0
\end{pmatrix}.
\end{equation}

$(ii)$ The functional \eqref{Hamiltonian}
\begin{equation*}\mathcal{H}(\xi,\chi):=\frac12\int_{\T^1}\chi\, G(\xi)\chi\,d\th + \s_0\int_{\T^1}e^\xi\sqrt{1+\xi'^2}\,d\th - \frac{\s_0}{2}\int_{\T^1}e^{2\xi}\,d\th,
\end{equation*}
is the Hamiltonian of the system \eqref{h.eq.}-\eqref{psi.eq.} with respect to the Poisson brackets \eqref{Poisson.brackets}; in other words, \eqref{h.eq.}-\eqref{psi.eq.} is equivalent to the system
\begin{equation}\label{Hamiltonian.system}\pa_t \xi = e^{-2\xi}\pa_\chi\mathcal{H}(\xi,\chi),\qquad\pa_t\chi = -e^{-2\xi}\pa_\xi\mH(\xi,\chi).
\end{equation}
Moreover, for any $u:=(\xi,\chi)$ one has
\begin{align}&\grad\mH(u)=(\pa_\xi\mH(u),\pa_\chi\mH(u)), \label{grad.H}
\\&\pa_\xi\mH(u)=-\frac12\Big(\frac{G(\xi)\chi + \xi'\chi'}{\sqrt{1+\xi'^2}}\Big)^2 + \frac12\chi'^2 - \s_0\Big(\Big[\frac{\xi'}{\sqrt{1+\xi'^2}}\Big]'-\frac{1}{\sqrt{1+\xi'^2}}+e^{2\xi}\Big), \label{pa.xi.H}
\\&\pa_\chi\mH(u)=G(\xi)\chi,
\label{pa.chi.H}
\end{align}
and $\mH$ is constant in time along the solutions the system \eqref{Hamiltonian.system}.
\end{lemma}

\begin{proof} $(i)$ Straightforward computations.

$(ii)$ Let us start by proving the first identity in \eqref{Hamiltonian.system}. For any $(\xi,\chi)$ and for any direction $\hat{\chi}$ corresponding to the second variable in $\mH$, we have
\begin{align*}d\mH(\xi,\chi)\hat\chi=\int_{\T^1} \hat\chi\cdot G(\xi)\chi\,d\th,
\end{align*}
so the $L^2-$gradient of $\mH$ with respect to the $\chi$ variable is
\begin{equation*}\pa_\chi\mH(\xi,\chi)= G(\xi)\chi.
\end{equation*}
Let us prove now \eqref{pa.xi.H}.
For any $(\xi,\chi)$ and for any direction $\hat{\xi}$ corresponding to the first variable in $\mH$, by Lemma \ref{lemma:G}, equations \eqref{shape.derivative} and \eqref{B.V.}, and by some integration by parts we get
\begin{align*}d\mH(\xi,\chi)\hat\xi&=\frac12\int_{\T^1}\chi(d G(\xi)[\hat\xi])\chi\,d\th + \s_0\int_{\T^1}\Big[e^\xi\hat\xi\sqrt{1+\xi'^2} + \frac{e^\xi\xi'\hat\xi'}{\sqrt{1+\xi'^2}}-e^{2\xi}\hat\xi\Big]\,d\th
\\&=\frac12\int_{\T^1}\chi(-G(\xi)[B\hat{\xi}]-(V\hat\xi)') \,d\th 
\\&\qquad+ \s_0\int_{\T^1}\Big(-e^{2\xi}+e^\xi\sqrt{1+\xi'^2}-\Big(\frac{e^\xi\xi'}{\sqrt{1+\xi'^2}} \Big)' \Big)\hat\xi\,d\th
\\&=\frac12\int_{\T^1}-(B\cdot G(\xi)\chi-V\chi')\hat\xi \,d\th 
\\&\qquad+ \s_0\int_{\T^1}\Big(-e^{2\xi}+e^\xi\sqrt{1+\xi'^2}-e^\xi\frac{\xi'^2}{\sqrt{1+\xi'^2}}-e^\xi\Big( \frac{\xi'}{\sqrt{1+\xi'^2}}\Big)'\Big)\hat\xi\,d\th
\\&=\frac12\int_{\T^1}-(B\cdot G(\xi)\chi-V\chi')\hat\xi \,d\th 
\\&\qquad+ \s_0\int_{\T^1}\Big(-e^{2\xi}+e^\xi\frac{1}{\sqrt{1+\xi'^2}}-e^\xi\Big( \frac{\xi'}{\sqrt{1+\xi'^2}}\Big)'\Big)\hat\xi\,d\th
\end{align*}
from which we get
\begin{align*}\pa_\xi\mH(\xi,\chi)&=-\frac12 B\cdot G(\xi)\chi + \frac12 V\chi' - \s_0\Big(\Big[\frac{\xi'}{\sqrt{1+\xi'^2}}\Big]'-\frac{1}{\sqrt{1+\xi'^2}}+e^{2\xi}\Big)
\\&=-\frac12\Big[\frac{(G(\xi)\chi)^2+G(\xi)\chi\cdot\xi'\chi'}{1+\xi'^2}-\chi'^2+\frac{G(\xi)\chi\cdot\chi'\xi'+\xi'^2\chi'^2}{1+\xi'^2}\Big]
\\&\qquad\qquad- \s_0\Big(\Big[\frac{\xi'}{\sqrt{1+\xi'^2}}\Big]'-\frac{1}{\sqrt{1+\xi'^2}}+e^{2\xi}\Big)
\\&=-\frac12\Big(\frac{G(\xi)\chi + \xi'\chi'}{\sqrt{1+\xi'^2}}\Big)^2 + \frac12\chi'^2 - \s_0\Big(\Big[\frac{\xi'}{\sqrt{1+\xi'^2}}\Big]'-\frac{1}{\sqrt{1+\xi'^2}}+e^{2\xi}\Big).
\end{align*}

\end{proof}

Thanks to the Hamiltonian structure, we can derive constants of motion from symmetries.
To start with, let us derive the conservation of the angular momentum from translation invariance:
\begin{lemma}\label{lemma:rotation}
Let $\{\mT_\alpha\colon\,\alpha\in\T^1\}$ be the one-parameter group of translation operators, whose action on each $(\xi,\chi)$ is defined as
\begin{equation}\label{translation.op}(\mT_\alpha(\xi,\chi))(\th):=(\xi(\th+\alpha),\chi(\th+\alpha)),\qquad\forall\th\in\T^1.
\end{equation}
Then,

$(i)$ the infinitesimal generator $\frac{d}{d\alpha}\big|_{\alpha=0}\mT_\alpha$ of the one-parameter group $(\mT_\alpha)_{\alpha\in\T^{1}}$ is the derivative operator $\pa_\th$, which commutes with $\mT_\alpha$, that is, 
\begin{equation}\label{D.T.commutation}\mT_\alpha\circ \pa_\th=\pa_\th\circ\mT_\alpha;
\end{equation}
$(ii)$ the Hamiltonian $\mH$ is $\mT_\alpha-$invariant and $\grad\mH$ is $\mT_\alpha-$equivariant, that is 
\begin{equation}\label{H.trans.inv}\mH\circ\mT_\alpha=\mH,\qquad\grad\mH\circ\mT_\alpha = \mT_\alpha\circ\grad\mH;
\end{equation}
$(iii)$ the vector field $(-e^{2\xi}\chi',e^{2\xi}\xi')$ is the $L^2-$gradient of the angular momentum \eqref{def.mI}
\begin{equation*}\mI(\xi,\chi):=-\frac12\int_{\T^1}e^{2\xi}\chi'\,d\th=\int_{\T^1}e^{2\xi}\xi'\chi\,d\th,\end{equation*}
which is a conserved quantity along the equations \eqref{Hamiltonian.system}, that is,
\begin{equation}\label{I.commutes.H}\{\mI,\mH\}=0;
\end{equation}
moreover, $\mI$ is $\mT_\alpha-$invariant and $\grad\mI$ is $\mT_\alpha-$equivariant, that is,
\begin{equation}\label{I.equiv}\mI\circ\mT_\alpha=\mI,\qquad\grad\mI\circ\mT_\alpha=\mT_\alpha\circ\grad\mI.
\end{equation}
    
\end{lemma}

\begin{proof} 

$(i)$ It is immediate to check. 

$(ii)$ The first identity in \eqref{H.trans.inv} follows by using Lemma \ref{lemma:G}, equation \eqref{G.translation.inv}, and by the change of variable $\beta=\th+\alpha$.

As for the second identity, for any $\alpha\in\T^1$, $u=(\xi,\chi)$ and $\hat u=(\hat\xi,\hat\chi)$, one has
\begin{equation*}d\mH(\mT_\alpha u)[\mT_\alpha\hat u]=d\mH(u)[\hat u],
\end{equation*}
which means
\begin{equation*}\la\grad\mH(\mT_\alpha u),\mT_\alpha\hat u\ra_{L^2}=\la\grad\mH(u),\hat u\ra,
\end{equation*}
but since $\la\grad\mH(\mT_\alpha u),\mT_\alpha\hat u\ra_{L^2} = \la\mT_\alpha^*\grad\mH(\mT_\alpha u),\hat u\ra_{L^2}$, where $\mT_{\alpha}^*$ is the $L^2-$adjoint operator associated to $\mT_{\alpha}$, then 
\begin{equation*}\mT_\alpha^*\circ\grad\mH\circ\mT_\alpha=\grad\mH.
\end{equation*}
By the change of variable $\beta=\th+\alpha$, one can observe that $\la\mT_\alpha f,g\ra_{L^2}=\la f,\mT_{-\alpha}g\ra=\la f,\mT_\alpha^{-1}g\ra$ for any $f,g\in L^2$, so $\mT_\alpha^*=\mT_\alpha^{-1}$, from which we get the second identity in \eqref{H.trans.inv}.

$(iii)$ It is immediate to check that $\grad\mI=(-e^{2\xi}\chi', e^{2\xi}\xi')$. By \eqref{H.trans.inv}, $\frac{d}{d\alpha}\mH(\mT_\alpha u)=0$ for any $u=(\xi,\chi)$, from which we get
\begin{equation*}0=\la\grad\mH(u),\frac{d}{d\alpha}\Big|_{\alpha=0}T_\alpha u\ra_{L^2}=\la\grad\mH(u),u'\ra_{L^2}=\la e^{-2\xi}J_0\grad\mH(u),\grad\mI(u)\ra_{L^2}=\{\mH,\mI\}(u).
\end{equation*}
The identities in \eqref{I.equiv} can be obtained exactly as \eqref{H.trans.inv}.

\end{proof}

Let us now derive the conservation of the fluid mass from translation invariance of the boundary velocity potential $\chi$:
\begin{lemma}\label{lemma:shift.mass}
For any $a\in\R$, let $S_a$ be the shift operator:
\begin{equation}\label{shift.op}S_a(\xi,\chi):=(\xi,\chi+a).
\end{equation}
Then,

$(i)$ the infinitesimal generator $\frac{d}{da}\big|_{a=0}S_a$ of the one-parameter group $\{S_a\colon\,a\in\R\}$ is the constant map $(h,\psi)\longmapsto(0,1)$;

$(ii)$ the Hamiltonian $\mH$ in \eqref{Hamiltonian} and its $L^2-$gradient are $S_a-$invariant, that is,
\begin{equation}\label{H.is.Sa.invariant}\mH\circ S_a=\mH,\qquad\grad\mH\circ S_a = \grad\mH;
\end{equation}

$(iii)$ the vector field $(e^{2\xi},0)$ is the $L^2-$gradient of the drop mass
\begin{equation}\label{mass}\mM(\xi):=\frac12\int_{\T^1}e^{2\xi}\,d\s,
\end{equation}
which is a conserved quantity along the equations \eqref{Hamiltonian.system}, that is,
\begin{equation}\label{M.is.conserved}\{\mM,\mH\}=0.
\end{equation}
Moreover, $\mM$ and $\grad\mM$ are both $S_a-$invariant and $\mT_\alpha-$equivariant (see \eqref{translation.op}), that is,
\begin{align}&\mM\circ S_a = \mM,\qquad\qquad\qquad\mM\circ\mT_\alpha=\mM, \label{M.is.invariant}
\\&\grad\mM\circ S_a= S_a \circ\grad\mM,\qquad\grad\mM\circ\mT_\alpha = \mT_\alpha\circ\grad\mM \label{grad.M.is equiv.}
\end{align}

\end{lemma}

\begin{proof}

$(i)$ It is immediate to check.

$(ii)$ The first identity can be obtained by using Lemma \ref{lemma:G}, equations \eqref{G.is.symmetric} and \eqref{G.is.zero.on.constants}, while the second one follows from observing that $\grad\mH(u)$ depends on $\chi$ only through its derivatives.

$(iii)$ It is immediate to check that $\grad\mM(u)=(e^{2\xi},0)$, so by using \eqref{G.is.symmetric} and \eqref{G.is.zero.on.constants} one has that for any $u=(\xi,\chi)$,
\begin{equation*}\{\mM,\mH\}(u)=\la 1,G(\xi)\chi\ra_{L^2}=0,
\end{equation*}
so we get \eqref{M.is.conserved}. All the identities in \eqref{M.is.invariant} and \eqref{grad.M.is equiv.} can be immediately checked.

\end{proof}

Now, we show the reversibility property of the found constants of motions:

\begin{lemma}\label{lemma:reversibility} Let $\mR$ be the reversibility operator, that is, for all $\xi,\chi$ and for all $\th\in\T^1$,
\begin{equation}\label{revers.op.}(\mR(\xi,\chi))(\th):=(\xi(-\th),-\chi(-\th)).
\end{equation}
Then, the Hamiltonian $\mH$, the angular momentum $\mI$ and the mass $\mM$ are $\mR-$invariant and their $L^2-$gradients are $\mR-$equivariant, that is,
\begin{align}&\mH\circ\mR=\mH,\qquad\qquad\qquad\mI\circ\mR=\mI,\qquad\quad\qquad\mM\circ\mR=\mM, \label{R.inv.}
\\&\grad\mH\circ\mR=\mR\circ\grad\mH,\qquad\grad\mI\circ\mR=\mR\circ\grad\mI,\qquad\grad\mM\circ\mR=\mR\circ\grad\mM \label{R.equiv}
\end{align}

\end{lemma}

\begin{proof} 

The identities in \eqref{R.inv.} follows by applying \eqref{G.inversion} and the change of variable $\beta=-\th$.

As for the second identities, one has that for any $u$ and for any direction $\hat u$,
\begin{equation*}d\mH(\mR u)[\mR\hat u]=d\mH(u)\hat u,
\end{equation*}
that is,
\begin{equation*}\mR^*\circ\grad\mH\circ\mR =\grad\mH.
\end{equation*}
Now, since for any $u,v\in L^2$, by the change of variable $\beta=-\xi$ we get
\begin{equation*}\la\mR u,v\ra_{L^2}=\la u,\mR v\ra_{L^2},
\end{equation*}
then $\mR^*=\mR$, from which we get the second identity in \eqref{R.equiv}.

\end{proof}

\section{Rotating waves}

In this section, we want to find solutions of \eqref{Hamiltonian.system} of the form
\begin{equation}\label{Rotating.waves}\xi(t,\th):=\eta(\th+\om t),\qquad\chi(t,\th):=\beta(\th+\om t),
\end{equation}
where $\xi,\eta\colon\,\T^1\longmapsto\R$, $t\in\R$, $\th\in\T^1$ and $\om\in\R$. We call these solutions \emph{rotating waves}. One of them is the \emph{static circle solution} $(h,\psi)=(0,0)$.

The main results of this section are the following theorems. The first concernes the existence of rotating waves which are not symmetric \emph{a priori} and close to the static circle.
\begin{theorem} \label{thm:main} 
Let $\mathfrak{s}\ge0,\,s_0>0$, let $\om_*$ be defined in \eqref{deg.freq.}. 
Let $s \ge s_0 > 1$.
Then there exist $a_0 > 0$, $C > 0$ such that, for every $a \in (0, a_0)$, 
there exists one orbit
\begin{equation}  \label{n.orbits}
\{ (\eta_a , \beta_a )\circ\mT_\alpha : \alpha \in \T^1 \}
\end{equation}
of solutions of the rotating traveling wave equations \eqref{rotating.wave.equation.0}
with angular velocity $\om = \om_a$, 
angular momentum 
\begin{equation} \label{ang.mom.a}
\mI(\eta_a , \beta_a) = a,
\end{equation}
and 
\begin{equation} \label{est.sol}
|\om_a - \om_*| + \| \eta_a \|_{H^{s+\frac32}} + \| \beta_a \|_{H^{s+1}} 
\leq C \sqrt{a}.
\end{equation}
Moreover, the orbit depends analytically on $a$ in the interval $(0, a_0)$. 
\end{theorem}
In Section 4.4, we also provide the existence of symmetric rotating waves close to the static circle. To this term, let us introduce the space of the fixed points for the reversibility operator $\mR$
\begin{equation}\label{R.fix.point}\begin{aligned}E:&=\{u=(\eta,\beta)\in(L^2)^2\colon\,\mR u=u\}
\\&=\{u=(\eta,\beta)\in(L^2)^2\colon\,(\eta(\th),\beta(\th))=(\eta(-\th),-\beta(-\th)),\,\forall\th\in\T^1\},
\end{aligned}\end{equation}
and the space of functions with $c-$fold symmetry:
\begin{equation}\label{Sobolev.c.fold}\begin{aligned}L_c^2:&=\Big\{f\in L^2\colon\,\mT_{\frac{2\pi}{c}} f=f\Big\}
\\&=\Big\{f\in L^2\colon\,f\Big(\th+\frac{2\pi}{c}\Big)=f(\th),\quad\forall\th\in\T^1\Big\}\end{aligned}\end{equation}

\begin{theorem} \label{thm:bif}

In the same regularity assumptions as in Theorem \ref{thm:main}, the following facts hold.

\medskip

$(i)$ (Solutions even in $\eta$, odd in $\beta$)

Let $\om_*$ be defined in \eqref{deg.freq.}. 
Then, given $E$ as in \eqref{R.fix.point} and $V_E, W_E$ as in \eqref{V.E}, one has that $\om_*$ is a bifurcation point from a simple eigenvalue, and the set of the solutions of \eqref{rotating.wave.equation.0} restricted to $E$ close to $(\om_*,0)$ in $\R\times E$ is a unique analytic curve parametrized on the $1-$dimensional space $V_E$.

$(ii)$ (Solutions with $c-$fold symmetry)

Let $c\ge 2$ be an integer.
Then, there exist $a_0 > 0$, $C > 0$ such that, for every $a \in (0, a_0)$, 
there exists one orbit 
\begin{equation}\{(\eta_a,\beta_a)\circ\mT_\alpha\colon\,\alpha\in\T^1\}
\end{equation}
of solutions of the rotating traveling wave equations \eqref{rotating.wave.equation.0} which belong to $L_c^2\times L_c^2$, have angular velocity $\om = \om_a$, 
angular momentum
\begin{equation}\mI(\eta_a , \beta_a) = a,\end{equation}
and satisfy
\begin{equation*} 
|\om_a - \om_*| + \| \eta_a \|_{H^{s+\frac32}} + \| \beta_a \|_{H^{s+1}} 
\leq C \sqrt{a}.
\end{equation*}
Moreover, the solutions depend analytically on $a$ in the interval $(0, a_0)$.

$(iii)$ (Solutions even in $\eta$, odd in $\beta$, with $c-$fold symmetry)

Let $c\ge2$ be an integer, let $\om_*$ be defined in \eqref{freq.c.fold}. 
Then, given $E$ as in \eqref{R.fix.point}, $L_c^2$ as in \eqref{Sobolev.c.fold} and $V_E,\,W_E$ as in \eqref{V.E}, one has that $\om_*$ is a bifurcation point from a simple eigenvalue, and the set of the solutions of \eqref{rotating.wave.equation.0} restricted to $E$ close to $(\om_*,0)$ in $\R\times E$ is a unique analytic curve parametrized on the $1-$dimensional space $V_E\cap(L_c^2\times L_c^2)$.

\end{theorem}

\begin{remark}\label{remark.Moon} While Theorem \ref{thm:main} and Theorem \ref{thm:bif}, points $(i)-(ii)$ are new results, the point $(iii)$ is partially the result proved in \cite{Moon.Wu}: the authors actually provided a global bifurcation result, while ours is a local one, since we require $\mF$ to be analytic close to the static circle solution $(\eta,\beta)=(0,0)$ and to this term, we need the condition $\|\eta\|<\delta_0<\infty$ which prevents us from applying the Rabinowitz Global Bifurcation Theorem, see Lemma \ref{lemma:G} and the definition of $U$ in Lemma \ref{ww.rotating.wave.eq}. 

\end{remark}

\subsection{The rotating wave equations}

Now, we derive the rotating wave equations:

\begin{lemma}\label{ww.rotating.wave.eq} Let us consider $\xi(t,\th),\chi(t,\th)$, $t\in\R,\th\in\T^1$, satisfying the ansatz \eqref{Rotating.waves}, let $\mathfrak{s}>0,\,s_0, s \in \R$, $s \geq s_0 > 1$, and let 
\[
U := \{ u = (\eta, \beta) : 
\eta \in H^{\mathfrak{s},s+\frac32}, \ \ 
\beta \in H^{\mathfrak{s},s+1}, \ \ 
\| \eta \|_{H^{\mathfrak{s},s_0 + 1}} < \delta_0 \},
\]
where $\delta_0$ is the constant in Lemma \ref{lemma:G}, point $(vii)$.
If we denote by
\begin{align} 
&\mF_0(\om; u) := (\mF_{1,0}(\om; u) , \mF_{2,0}(\om; u)),
\label{def.mF}
\\
&\mF_{1,0}(\om; u) := 
\om\eta' 
- e^{-2\eta}G(\eta) \beta,
\label{def.mF.1} 
\\
&\mF_{2,0}(\om; u) := 
\om\beta' 
- e^{-2\eta}\frac12 \Big( \frac{G(\eta) \beta + \eta'\beta'}{ 
\sqrt{1+\eta'^2}} \Big)^2
+ \frac12 e^{-2\eta}\beta'^2 
+ \s_0\Big(e^{-\eta}\Big[\frac{\eta'}{\sqrt{1+\eta'^2}} \Big]' - \frac{e^{-\eta}}{\sqrt{1+\eta'^2}} + 1  \Big), 
\label{def.mF.2} 
\end{align} 
then $\mF_{1,0}(\om; u) \in H^{\mathfrak{s},s}$, 
$\mF_{2,0}(\om; u) \in H^{\mathfrak{s},s-\frac12}$ 
for all $u \in U$, $\om \in \R$, the map 
\[
\mF : \R \times U \to H^{\mathfrak{s},s} \times H^{\mathfrak{s},s-\frac12}
\]
is analytic and if $u(t,\th)=(\eta(\th+\om t),\beta(\th+\om t))$ solves the system \eqref{Hamiltonian.system}, then 
\begin{equation}\label{rotating.wave.equation.0}\mF_0(\om;\eta,\beta)=(0,0).\end{equation}

\end{lemma}

\begin{proof} The analiticity of $\mF$ follows by Lemma \ref{lemma:G}, point $(vii)$. If we suppose that $u(t,\th)=(\eta(\th+\om t),\beta(\th+\om t))$ solves \eqref{Hamiltonian.system}, then $\pa_t\eta(\th+\om t)=\om\eta'(\th+\om t)$, $\pa_t\beta(\th+\om t)=\om\beta'(\th+\om t)$ and by \eqref{G.translation.inv}, $G(\eta((\cdot)+\om t))[\beta((\cdot)+\om t)]=(G(\eta)\beta)(\th+\om t)$, from which necessarily \eqref{rotating.wave.equation.0} holds.
\end{proof}

We observe now that the equation \eqref{rotating.wave.equation.0} is equivalent to a critical point equation of a suitable functional.
\begin{lemma}\label{lemma:variational.structure.rw}
Given the assumptions of Lemma \ref{ww.rotating.wave.eq}, we have that $u=(\eta,\beta)$ solves $\mF_0(\om,u)=0$ if and only if $u$ is a constrained critical point for $\mH$ under the constraint $\{u\in U\colon\,\mI(u)=a\}$ for some $a\in\R$, or equivalently, $u$ is a free critical point for the functional \eqref{lagrangian.functional}, which is
\begin{align*}\mE(\om;u):=\mH(u)-\om(\mI(u)-a).\end{align*}

\end{lemma}

\begin{proof} It is enough to recall Theorem \eqref{thm:Hamiltonian.structure}, point $(ii)$, equations \eqref{grad.H}-\eqref{pa.chi.H}, and Lemma \ref{lemma:rotation}, point $(iii)$: indeed, one finds out that $\grad\mE(\om;u)=0$ if and only if $e^{2\eta}\mF_0(\om;u)=0$.
    
\end{proof}

From now on, we will consider the nonlinear operator \begin{align}&\mF(\om;u):=\grad\mE(\om;u)=(\mF_1(\om;u),\mF_2(\om;u)),\label{F.op.}
\\&\mF_1(\om;u):=-\frac12 \Big( \frac{G(\eta) \beta + \eta'\beta'}{ 
\sqrt{1+\eta'^2}} \Big)^2
+ \frac12\beta'^2 
- \s_0\Big(e^{\eta}\Big[\frac{\eta'}{\sqrt{1+\eta'^2}} \Big]' - \frac{e^{\eta}}{\sqrt{1+\eta'^2}}+e^{2\eta}  \Big) + \om e^{2\eta}\beta',      \label{F.1}
\\&\mF_2(\om;u) := G(\eta)\beta - \om e^{2\eta}\eta' \label{F.2}
\end{align}
instead of $\mF_0$, so we will look for the zeros of $\mF$.

Let us notice that since the boundary velocity potential $\chi$ is defined up to constants, then the image of $\mF$ does not contain the subspace of $L^2\times L^2$ generated by $u_0:=(0,1)$. More precisely, by \eqref{I.commutes.H} and \eqref{M.is.conserved} we directly get the following lemma.
\begin{lemma}\label{lemma:v00.orth.property} 
In the assumptions of Lemma \ref{ww.rotating.wave.eq}, one has that for all $(\om,u)\in\R\times U$,
\begin{equation}\label{v00.orth.}\la(0,1),\mF(\om;u)\ra_{(L^2)^2}=0.
\end{equation}
\end{lemma}

In terms of symmetry properties of $\mF$, we have the following lemma, which is a direct consequence of Lemma \ref{lemma:rotation}, equations \eqref{H.trans.inv} and \eqref{I.equiv}, and Lemma \ref{lemma:reversibility}, equations \eqref{R.inv.} and \eqref{R.equiv}:
\begin{lemma}\label{lemma:F.invariance}
In the assumptions of Lemma \ref{ww.rotating.wave.eq}, we get that
\begin{align}&\mE\circ\mT_\alpha=\mE,\qquad\quad\quad\quad\mE\circ\mR=\mE,  \label{E.invariance}
\\&\mF\circ\mT_\alpha = \mT_\alpha\circ\mF,\qquad\mF\circ\mR = \mR\circ\mF.\label{F.equiv.}
\end{align}
    
\end{lemma}

\subsection{The linearized operator at static circle solution}

In this section, we study the algebraic properties of the linear operator $\mL_\om:=\pa_u\mF(\omega;0)$, which is
\begin{equation}\label{lin.op.}\mL_\om=\begin{pmatrix}-\s_0(1+\pa_{\th\th}^2) & \om\pa_\th \\ -\om\pa_\th & G(0)\end{pmatrix}
\end{equation}
Let us consider the Fourier expansion of $\eta,\beta$ by using the basis \eqref{Fourier.basis}:
\begin{equation}\label{Fourier.eta.beta}\eta=\sum_{(\ell,m)\in\mT}\eta_{\ell,m}\ph_{\ell,m},\qquad\beta=\sum_{(\ell,m)\in\mT}\beta_{\ell,m}\ph_{\ell,m};
\end{equation}
starting from this, we want to compute its kernel and its range. 
One gets that
\begin{equation}\label{Lin.basis.action}\mL_\om(\eta,\beta)=\sum_{(\ell,m)\in\mT}\begin{pmatrix}r_{\ell,m}\ph_{\ell,m}\\ \rho_{\ell,-m}\ph_{\ell,-m}
\end{pmatrix},
\end{equation}
where
\begin{equation}\label{matrix.L.rep}\begin{pmatrix}r_{\ell,m}\\ \rho_{\ell,-m}
\end{pmatrix}:=\mL_{\ell,m}\begin{pmatrix}\eta_{\ell,m}\\ \beta_{\ell,-m}
\end{pmatrix}:=\begin{pmatrix}\s_0(\ell^2-1) & -\om m\ell \\ -\om m\ell & \ell
\end{pmatrix} \begin{pmatrix}\eta_{\ell,m}\\ \beta_{\ell,-m}
\end{pmatrix}.
\end{equation}
To determine the kernel of $\mL_\om$, it is enough to determine the kernel of each $\mL_{\ell,m}$. It is easy to see that given any $(\ell,m)\in\mT$, $\mL_{\ell,m}$ has a nonzero kernel if and only if $(\ell,m)=(0,0)$ or for some $\ell\ge1$ it holds
\begin{equation}\label{deg.eq.}\s_0(\ell^2-1)-\om^2\ell=0,
\end{equation}
or else, restricting only to the frequencies $\omega\ge0$, $\omega=\sqrt{\s_0}\cdot\sqrt{\tfrac{\ell^2 - 1}{\ell}}.$ 

From now on, we will suppose that there exists some $\ell_*\ge1$ such that
\begin{equation}\label{deg.freq.}\omega=\omega_*:=\sqrt{\s_0}\cdot\sqrt{\tfrac{\ell_*^2 - 1}{\ell_*}}\end{equation}
The following lemma concerns the maximum number of solutions for \eqref{deg.eq.} when we suppose \eqref{deg.freq.}:
\begin{lemma}\label{lemma:2.most.solutions}
Given \eqref{deg.freq.}, the equation \eqref{deg.eq.} has exactly one solution in $\N$, and it is $\ell=\ell_*$.
\end{lemma}

\begin{proof} 
Obviously, the choice \eqref{deg.freq.} of $\om$ is such that $\ell_*$ solves \eqref{deg.eq.}. We note that $\ell=0$ never solves \eqref{deg.eq.}.
Let us now suppose to have another distinct solution $\ell_1\ge 1$. Then, $\ell_1$ satisfies the identity
\begin{equation*}\sqrt{\frac{\ell_*^2-1}{\ell_*}}=\sqrt{\frac{\ell_1^2-1}{\ell_1}};
\end{equation*}
if we square both sides and multiply them by $\ell_*\ell_1$, we get $\ell_1(\ell_*^2-1)=\ell_*(\ell_1^2-1)$, from which
\begin{equation*}\ell_1\ell_*=-1,
\end{equation*}
which contradicts the fact that $\ell_1,\ell_*\ge1$.

\end{proof}

We notice that only the matrices $\mL_{0,0}$, $\mL_{\ell_*,1}$ and $\mL_{\ell_*,-1}$ have nonzero kernel, given by
\begin{align*}&\ker\mL_{0,0}=\Big\{\lambda_{0,0}\begin{pmatrix}0 \\ 1
\end{pmatrix}\colon\,\lambda_{0,0}\in\R\Big\},
\\&\ker\mL_{\ell_*,m}=\Big\{\lambda_{\ell_*,m}\begin{pmatrix}1 \\ -\om_*m
\end{pmatrix}\colon\,\lambda_{\ell_*,m}\in\R\Big\},
\end{align*}
hence, calling 
\begin{equation}\mS:=\{(0,0),(\ell_*,1),(\ell_*,-1)\}.
\end{equation}
by Lemma \ref{lemma:2.most.solutions} we get that the kernel $V$ of $\mL_{\om_*}$ and its $L^2\times L^2-$orthogonal complement $W$ are respectively
\begin{align}&V=\Big\{\lambda_{0,0}\begin{pmatrix}0 \\ \ph_{0,0}
\end{pmatrix} + \lambda_{\ell_*,1}\begin{pmatrix}\ph_{\ell_*,1} \\ -\om_*\ph_{\ell_*,-1}
\end{pmatrix} + \lambda_{\ell_*,-1}\begin{pmatrix}\ph_{\ell_*,-1} \\ \om_*\ph_{\ell_*,1}
\end{pmatrix}\colon\,\lambda_{0,0},\lambda_{\ell_*,1},\lambda_{\ell_*,-1}\in\R \Big\},
\label{kernel.L.}
\\&W=\Big\{\begin{pmatrix}\eta \\ \beta
\end{pmatrix} = \lambda_{0,0}\begin{pmatrix}\ph_{0,0} \\ 0
\end{pmatrix} + \lambda_{\ell_*,1}\begin{pmatrix}\om_*\ph_{\ell_*,1}\\ \ph_{\ell_*,-1}
\end{pmatrix} + \lambda_{\ell_*,-1}\begin{pmatrix}-\om_*\ph_{\ell_*,-1}\\ \ph_{\ell_*,1}
\end{pmatrix} + \sum_{(\ell,m)\in\mT\setminus \mS}\begin{pmatrix}\eta_{\ell,m}\ph_{\ell,m}\\ \beta_{\ell,-m}\ph_{\ell,-m}\notag
\end{pmatrix}
\\&\qquad\qquad\qquad\qquad\qquad\colon\lambda_{0,0},\lambda_{\ell_*,1},\lambda_{\ell_*,-1},\eta_{\ell,m},\beta_{\ell,m}\in\R,\qquad(\eta,\beta)\in L^2\times L^2\Big\}.\label{W}
\end{align}
From this, we get the decomposition 
\begin{equation}\label{lyap.schm.}L^2\times L^2=V\oplus W,
\end{equation}
from which we get also
\begin{equation}\label{lyap.schm.dec.}H^{\mathfrak{s},s+1}\times H^{\mathfrak{s},s+\frac12}=V\oplus W^{\mathfrak{s},s},\qquad W^{\mathfrak{s},s}:=W\cap(H^{\mathfrak{s},s+\frac32}\times H^{\mathfrak{s},s+1}).
\end{equation}

To compute the range of $\mL_{\om_*}$, by using \eqref{Lin.basis.action} and \eqref{matrix.L.rep} we observe that $\mL_{\om_*}$ acts on $W$ as
\begin{align*}&\mL_{\om_*}\begin{pmatrix}\ph_{0,0} \\ 0
\end{pmatrix} = \begin{pmatrix}\s_0\ph_{0,0} \\ 0
\end{pmatrix},
\\&\mL_{\om_*}\begin{pmatrix}\om_*\ph_{\ell_*,1} \\ \ph_{\ell_*,-1}
\end{pmatrix} = -\ell_*(\om_*^2+1)\begin{pmatrix}\om_*\ph_{\ell_*,1}\\ \ph_{\ell_*,-1}
\end{pmatrix},
\\&\mL_{\om_*}\begin{pmatrix}-\om_*\ph_{\ell_*,-1} \\ \ph_{\ell_*,1}
\end{pmatrix} = \ell_*(\om_*^2+1)\begin{pmatrix}\om_*\ph_{\ell_*,-1}\\ -\ph_{\ell_*,1}
\end{pmatrix},
\\&\mL_{\om_*}\begin{pmatrix}\eta_{\ell,m}\ph_{\ell,m} \\ \beta_{\ell,-m}\ph_{\ell,-m}
\end{pmatrix} = \begin{pmatrix}(-\s_0(\ell^2-1)\eta_{\ell,m}-\om_* m\ell\beta_{\ell,-m})\ph_{\ell,m} \\ (-\om_* m\ell\eta_{\ell,m}-\ell\beta_{\ell,-m})\ph_{\ell,-m}
\end{pmatrix},
\end{align*}
from which we get that the range $R$ of $L$ and its $L^2\times L^2-$orthogonal complement $Z$ are respectively
\begin{align}&R=\Big\{\begin{pmatrix}r \\ \rho
\end{pmatrix} = r_{0,0}\begin{pmatrix}\ph_{0,0} \\ 0
\end{pmatrix} + r_{\ell_*,1}\begin{pmatrix}\om_*\ph_{\ell_*,1} \\ \ph_{\ell_*,-1}
\end{pmatrix} +r_{\ell_*,-1}\begin{pmatrix}\om_*\ph_{\ell_*,-1} \\ -\ph_{\ell_*,1}
\end{pmatrix} + \sum_{(\ell,m)\in\mT\setminus\mS}\begin{pmatrix}r_{\ell,m}\ph_{\ell,m} \\ \rho_{\ell,-m}\ph_{\ell,-m}
\end{pmatrix} \notag
\\&\qquad\qquad\qquad\colon r_{0,0},r_{\ell_*,1},r_{\ell_*,-1},r_{\ell,m},\rho_{\ell,m}\in\R,\quad(r,\rho)\in L^2\times L^2
\Big\}, \label{R}
\\& Z=\Big\{r_{0,0}\begin{pmatrix}0 \\ \ph_{0,0}
\end{pmatrix} + r_{\ell_*,1}\begin{pmatrix}\ph_{\ell_*,1} \\ -\om_*\ph_{\ell_*,-1}
\end{pmatrix} + r_{\ell_*,-1}\begin{pmatrix}\ph_{\ell_*,-1} \\ \om_*\ph_{\ell_*,1}
\end{pmatrix}\colon\,r_{0,0},r_{\ell_*,1},r_{\ell_*,-1}\in\R\Big\} = V
\end{align}
As well as done in \eqref{lyap.schm.}-\eqref{lyap.schm.dec.}, we have the decomposition
\begin{equation}\label{range.split}L^2\times L^2=R\oplus Z,\end{equation}
and so
\begin{equation}\label{range.dec}H^{\mathfrak{s},s-\frac12}\times H^{\mathfrak{s},s}=R^{\mathfrak{s},s}\oplus Z,\qquad R^{\mathfrak{s},s}:=R\cap(H^{\mathfrak{s},s-\frac12}\times H^{\mathfrak{s},s}).\end{equation}

Now, for any $(\eta,\beta)\in W^{\mathfrak{s},s}$ we set $\|(\eta,\beta)\|_{W^{\mathfrak{s},s}}^2:=\|\eta\|_{H^{\mathfrak{s,}s+\frac32}}^2 + \|\beta\|_{H^{\mathfrak{s},s+1}}^2$, while for any $(f,g)\in R^{\mathfrak{s},s}$ we set $\|(f,g)\|_{R^{\mathfrak{s},s}}^2:=\|f\|_{H^{\mathfrak{s},s-\frac12}}^2 + \|g\|_{H^{\mathfrak{s},s}}^2$
Next, we prove that the restriction of $\mL_{\om_*}$ on $W^{\mathfrak{s},s}$ is a linear homeomorphism between $W^{\mathfrak{s},s}$ and $R^{\mathfrak{s},s}$.

\begin{lemma}\label{L.invertibility} In the notations above, the linear operator $\mL_{\om_*}|_{W^{\mathfrak{s},s}}\colon W^{\mathfrak{s},s}\longmapsto R^{\mathfrak{s},s}$ is well-defined and bijective; in particular, it satisfies the estimate
\begin{equation}\label{L.homeo.est.}\|\mL_{\om_*}(\eta,\beta)\|_{R^{\mathfrak{s},s}}\le C(\om_*,\s_0)\|(\eta,\beta)\|_{W^{\mathfrak{s},s}},\qquad\forall(\eta,\beta)\in W^{\mathfrak{s},s},
\end{equation}
where $C(\om_*,\s_0)>0$ is a constant depending on $\om_*,\s_0$.

As a result, the inverse operator of $\mL_{\om_*}$, namely $L|_{W^{\mathfrak{s},s}}^{-1}\colon R^{\mathfrak{s},s}\longmapsto W^{\mathfrak{s},s}$, is continuous.
\end{lemma}

\begin{proof} We prove \eqref{L.homeo.est.} only for $\mathfrak{s}=0$, as in the other cases the proof follows identically. Fixed $(\eta,\beta)\in W^s$, we get
\begin{align*}\|\mL_{\om_*}(\eta,\beta)\|_{R^s}^2 & = \sum_{(l,m)\in\mT\setminus\mS}\ell^{2(s-\frac12)}(\s_0(\ell^2-1)\eta_{\ell,m} + \om_* m \ell\b_{\ell,-m})^2 
\\&\qquad\qquad+ \sum_{(l,m)\in\mT\setminus\mS}\ell^{2s}(\om_* m \ell\eta_{\ell,m} + \ell\b_{\ell,-m})^2
\\&\le 2\sum_{(l,m)\in\mT\setminus\mS}\ell^{2(s-\frac12)}(\s_0^2(\ell^2-1)^2\eta_{\ell,m}^2 + \om_*^2\ell^2\b_{\ell,-m}^2) 
\\&\qquad\qquad+ 2\sum_{(l,m)\in\mT\setminus\mS}\ell^{2s}(\om_*^2\ell^2\eta_{\ell,m}^2 + \ell^2\b_{\ell,-m}^2)
\\&=2\sum_{(\ell,m)\in\mT\setminus\mS}[\s_0^2(\ell^4 - 2\ell^2 + 1) + \om_*^2\ell^{3}]\ell^{2(s-\frac12)}\eta_{\ell,m}^2 
\\&\qquad\qquad+ 2\sum_{(\ell,m)\in\mT\setminus\mS}[\om_*^2+\ell]\ell^{2(s+\frac12)}\b_{\ell,-m}^2
\\&\le2(\s_0^2+\om_*^2+3)\|\eta\|_{H^{s+\frac32}}^2 + 2(\om_*^2+1)\|\beta\|_{H^{s+1}}^2
\\&\le2(\s_0^2 + \om_*^2 + 3)\|(\eta,\beta)\|_{W^s}^2;
\end{align*}
we have used the definition \eqref{R} of $R$, the inequality $(|a|+|b|)^2\le2(a^2+b^2)$ for $a,b\in\R$, and the fact that for all $(\ell,m)\in\mT\setminus\mS$, one has $\ell\ge1$.
Finally we get \eqref{L.homeo.est.} with $C(\om_*,\s_0):=\sqrt{2(\s_0^2 + \om_*^2 +3)}$.

Hence, we get that $\mL_{\om_*}|_{W^s}$ maps $W^s$ to $R^s$, and it is a bijection by definitions \eqref{W} of $W$ and \eqref{R} of $R$; finally, by Open Mapping Theorem the inverse of $\mL_{\om_*}|_{W^s}$ is continuous.
\end{proof}

Let us now further decompose $V$ as $V=V_N\oplus V_D$, where
\begin{equation}\label{nondeg.space}\begin{aligned}&V_N:=\Big\{\lambda_{1}\mathtt v_{1} + \lambda_{-1}\mathtt v_{-1}\colon\,\lambda_1,\lambda_{-1}\in\R \Big\},
\\&\mathtt v_{m}:=(1+\om_*^2)^{-\frac12}\begin{pmatrix}\ph_{\ell_*,m}\\-\om_*m\ph_{\ell_*,-m}
\end{pmatrix},\qquad \mathtt v_0:=\begin{pmatrix}0 \\ \ph_{0,0}
\end{pmatrix}.
\end{aligned}\end{equation}
We call then $Z_N:=V_N$.
In these notations, we can prove that the torus action $\mT_\alpha$ is compatible with the decompositions \eqref{lyap.schm.}:
\begin{lemma}\label{Torus.action.lyap.schmidt}
In the notations above, for all $\alpha\in\T^1$ we have the following identities:
\begin{equation}\label{torus.action.v.var}\mT_\alpha\mathtt v_0=\mathtt v_0,\qquad\mT_\alpha\begin{pmatrix}\mathtt v_{1} \\ \mathtt v_{-1}\end{pmatrix} = \begin{pmatrix}\cos(\ell_*\alpha) & -\sin(\ell_*\alpha) \\ \sin(\ell_*\alpha) & \cos(\ell_*\alpha)
\end{pmatrix}\begin{pmatrix}\mathtt v_1 \\ \mathtt v_{-1}
\end{pmatrix}.
\end{equation}
As a result, one has
\begin{equation}\label{torus.inv.lyap.schmidt}\mT_\alpha V_N=V_N,\qquad\mT_\alpha V=V,\qquad\mT_\alpha W=W.
\end{equation}
    
\end{lemma}

\begin{proof}
Let us start by proving \eqref{torus.action.v.var}. 
The first identity is obvious, being $\mathtt v_0$ a constant vector. As for the second identity, by \eqref{nondeg.space} and by \eqref{Fourier.basis} one has
\begin{align*}(\mT_\alpha\mathtt v_1)(\th)&=(1+\om_*^2)^{-\frac12}\begin{pmatrix}\ph_{\ell_*,1}(\th+\alpha) \\ -\om_*\ph_{\ell_*,-1}(\th+\alpha)
\end{pmatrix}
\\&= (1+\om_*^2)^{-\frac12}\begin{pmatrix}\cos(\ell_*\th+\ell_*\alpha) \\ -\om_*\sin(\ell_*\th+\ell_*\alpha)
\end{pmatrix}
\\&= (1+\om_*^2)^{-\frac12}\begin{pmatrix}\cos(\ell_*\alpha)\cos(\ell_*\th) - \sin(\ell_*\alpha)\sin(\ell_*\th) \\ -\om_*[\cos(\ell_*\alpha)\sin(\ell_*\th) + \sin(\ell_*\alpha)\cos(\ell_*\th)]\end{pmatrix}
\\&=\cos(\ell_*\alpha)\mathtt v_1(\th) - \sin(\ell_*\alpha)\mathtt v_{-1}(\th).
\end{align*}
An analogous computation shows that $\mT_\alpha\mathtt v_{-1}=\sin(\ell_*\alpha)\mathtt v_1 + \cos(\ell_*\alpha)\mathtt v_{-1}$, from which the second identity in \eqref{torus.action.v.var}.

From \eqref{torus.action.v.var} we immediately get the first and the second identity in \eqref{torus.inv.lyap.schmidt}. The third identity can be obtained as follows.
Let $w\in W$. Then, for any $\alpha\in\T^1$ and for any $v\in V$, since $\mT_\alpha^*=\mT_{-\alpha}=(\mT_\alpha)^{-1}$ and $W\perp V=\mT_\alpha V$ one has
\begin{equation*}\la\mT_\alpha w,v\ra_{(L^2)^2}=\la w, T_{-\alpha}v\ra_{(L^2)^2}=0,
\end{equation*}
from which $\mT_\alpha w\in W$, and so $\mT_\alpha W\subseteq W$; however, by applying $\mT_{-\alpha}$ to this inclusion, one has $W\subseteq T_{-\alpha} W$ but also $\mT_{-\alpha}W\subseteq W$, so that $\mT_{-\alpha}W=W$ and, of course, $\mT_\alpha W=W$.

\end{proof}

\subsection{Existence of rotating waves}

In this section, we want to find nonzero solutions to the rotating wave equation
\begin{equation}\label{rw.eq.}\mF(\om;u)=0.
\end{equation}
In particular, we want to show that each $\omega_*$ as in \eqref{deg.freq.} is a bifurcation point from a multiple eigenvalue. As in the section before, we will denote by $|\cdot|$ the Euclidean norm on a generic finite-dimensional vector space, and by $\|\cdot\|:=\|\cdot\|_{W^{\mathfrak{s},s}}$ the norm of $W^{\mathfrak{s},s}$. All the proofs of this section are very similar to \cite{BLS}; however, for a matter of self-containedness, we will write them.

By the decomposition \eqref{lyap.schm.dec.}, each $u\in U$ can be uniquely written as
\begin{equation}\label{u.dec.}u=v+w,\qquad v\in V,\,w\in W^{\mathfrak{s},s},
\end{equation}
and that, defining the projection maps $\Pi_{R^{\mathfrak{s},s}}\colon H^{\mathfrak{s},s-\frac12}\times H^{\mathfrak{s},s}\longmapsto R^{\mathfrak{s},s}$ and $\Pi_{Z}\colon H^{\mathfrak{s},s-\frac12}\times H^{\mathfrak{s},s}\longmapsto Z$, the equation \eqref{rw.eq.} is equivalent to the system of equations
\begin{align}&\Pi_{R^{\mathfrak{s},s}}\mF(\omega;v+w)=0, \label{Range.eq.}
\\&\Pi_{Z}\mF(\omega;v+w)=0.\label{Bif.eq.}
\end{align}
For any fixed $\epsilon>0$, let
\begin{equation}\label{U.eps}\mU_\epsilon:=\{(\om,v)\in\R\times V\colon\,|\om-\om_*|<\epsilon,\,|v|<\epsilon\}.
\end{equation}
We start by observing that the range equation \eqref{Range.eq.} can be solved by Implicit Function Theorem.
\begin{lemma}\label{Range.eq.solv.}Let $\mF$ be as in \eqref{F.def.} and let $\om_*$ satisfy \eqref{deg.freq.}.

Then, there exists $\epsilon_0>0$, such that for all $(\om,v)\in\mU_{\epsilon_0}$, there exists one and only one analytic function $(\om,v)\in\mU_{\epsilon_0}\longmapsto w(\om,v)\in W^{\mathfrak{s},s}$ such that 
\begin{equation}\label{zeros.implicit.func.}\Pi_{R^{\mathfrak{s},s}}\mF(\om;v+w(\om;v))=0.\end{equation}
Moreover, for all $(\om,v)\in\mU_{\epsilon_0}$ and $\alpha\in\T^1$,
\begin{align}&w(\om,0)=0,\qquad\pa_v w(\om_*,0)=0, \label{w.values}
\\&w(\om;v)=O(|v|^2+|v|\cdot|\om-\om_*|),\label{w.est.}
\\&w(\om;\mT_\alpha v) = \mT_\alpha w(\om;v). \label{w.equiv.}
\end{align}

\end{lemma}

\begin{proof} The solvability of the range equation \eqref{Range.eq.} and the analiticity of $w(\om;v)$ in $\mU_{\epsilon_0}$ can be immediately deduced by Lemma \ref{L.invertibility}, thanks to which one can apply Implicit Function Theorem.
The first identity in \eqref{w.values} still follows from Implicit Function Theorem. The second identity follows from the fact that differentiating both sides of \eqref{zeros.implicit.func.} with respect to $v$ along any direction $\tilde v\in V$, one gets
\begin{equation*}0=\Pi_{R^{\mathfrak{s},s}}d\mF(\om_*;0)[\tilde v + dw(\om_*;0)\tilde v]=\Pi_{R^{\mathfrak{s},s}}\mL_{\om_*}|_{W^{\mathfrak{s},s}}[d w(\om_*;0)\tilde v],
\end{equation*}
which, by Lemma \ref{L.invertibility}, necessarily implies $d w(\om_*;0)\tilde v=0$ for all $\tilde v\in V$ and so $\pa_v w(\om_*;0)=0$.
The estimate \eqref{w.est.} follows by the Taylor expansion of $w$ about $(\om_*;0)$. In particular, since $\mF(\om;u)$ is linear in $\om$, necessarily $\pa_\om^k w(\om_*;0)=0$ for all $k\in\N$.
As for \eqref{w.equiv.}, by \eqref{zeros.implicit.func.} and by Lemma \ref{lemma:F.invariance}, equation \eqref{F.equiv.}, one has that for any $\alpha\in\T^1$ and for any $(\om,v)\in\mU_{\epsilon_0}$,
\begin{equation*}0=\mT_\alpha\Pi_{R^{\mathfrak{s},s}}\mF(\om;v+w(\om;v))=\Pi_{R^{\mathfrak{s},s}}\mT_\alpha\mF(\om;v+w(\om;v))=\Pi_{R^{\mathfrak{s},s}}\mF(\om;\mT_\alpha v+\mT_\alpha w(\om;v)).
\end{equation*}
Moreover, since $(\om,\mT_\alpha v)\in\mU_{\epsilon_0}$, we also have
\begin{equation*}\Pi_{R^{\mathfrak{s},s}}\mF(\om;\mT_\alpha v + w(\om;\mT_\alpha v))=0,
\end{equation*}
By uniqueness of $w(\om;v)$, one gets \eqref{w.equiv.}.

\end{proof}

Thanks to Lemma \ref{Range.eq.solv.}, we can reduce the problem of solving the equation $\mF(\om;u)=0$ to finding the solutions of the equation
\begin{equation}\label{red.eq.}\Pi_{Z}\mF(\omega;v+w(\om;v))=0.
\end{equation}
on $\mU_{\epsilon_0}$.
As a direct consequence of Lemma \ref{lemma:v00.orth.property}, in the bifurcation equation we can erase the projection onto the subspace of $Z$ generated by $\mathtt v_0$:
\begin{lemma}\label{lemma:reducing.nondeg.space}
In the assumptions of Lemma \ref{Range.eq.solv.}, one has that
\begin{equation}\label{nondeg.space.needed}\Pi_Z\mF(\om;v+w(\om;v))=0 \Longleftrightarrow\Pi_{Z_N}\mF(\om;v+w(\om;v))=0
\end{equation}
\end{lemma}
From now on, recalling the functional $\mE$ defined in \eqref{lagrangian.functional} we call
\begin{align}&\mU_{\epsilon_0}^{(N)}:=\mU_{\epsilon_0}\cap(\R\times V_N), \label{new.ball}
\\&f(\om;v):=\Pi_{Z_N}\mF(\om;v+w(\om;v)),
\qquad g(\om;v):=(\grad\mI)(v+w(\om;v)), \label{function.for.om(v)}
\\&F(\om;v):=\la f(\om;v),g(\om;v)\ra_{(L^2)^2} \label{F.def.}
\end{align}

We now want to perform a convenient choice of $\om=\om(v)$. To do this, we need the following technical lemma:

\begin{lemma}\label{lemma:technicalities} Let $(\om,v)\in\mU_{\epsilon_0}^{(N)}$. Then, the following facts hold.

$(i)$ $f\colon\,\mU_{\epsilon_0}^{(N)}\longmapsto Z_N$ and $g\colon\,\mU_{\epsilon_0}^{(N)}\longmapsto(L^2)^2$ are both analytic.

$(ii)$ One has 
\begin{align}
f(\om, 0) & = 0, 
\quad 
\pa_v f(\om_*,0) = 0, 
\quad 
\pa_{\om v} f(\om_*,0) 
= - \Pi_{Z_N} J_0 \pa_\th,
\label{der.f}
\\
g(\om, 0) & = 0, 
\quad 
\pa_v g(\om_*,0) 
= d (\grad \mI)(0) 
= J_0 \pa_\th
\label{der.g}
\end{align}
for any $|\om - \om_*| < \e_0$, 
where $J_0$ is defined in \eqref{J0}, 
and, 
\begin{align}
& F(\om, 0) = 0, 
\quad \ 
\pa_v F(\om, 0) = 0, 
\quad \ 
\pa_{vv} F(\om_*,0) = 0, 
\notag 
\\ 
& \pa_\om\{d^2F(\om_*,0) [\tilde v, \tilde v]\} 
= - | \Pi_{Z_N} J_0 \pa_\th \tilde v |^2 
\quad \forall \tilde v \in V.
\label{der.F}
\end{align}
$(iii)$ Let $v \in V$, with 
$v = v_{1}\mathtt v_{1} + v_{-1}\mathtt v_{-1}$, 
where $\mathtt{v}_{1},\mathtt v_{-1}$ are defined in \eqref{nondeg.space}. Then 
\begin{align}
&J_0 \pa_\th v 
= \ell_*(1+\om_*^2)^{-\frac12}\Big[v_{1}\begin{pmatrix}
\omega_* \ph_{\ell_*,1} \\ 
- \ph_{\ell_*,-1}
\end{pmatrix} + v_{-1}\begin{pmatrix}
\omega_* \ph_{\ell_*,-1} \\ 
 \ph_{\ell_*,1}
\end{pmatrix}\Big],
\label{formula.J0.mM.v} 
\\ 
&\Pi_{Z_N} J_0 \pa_\th v  
= 2\om_*\ell_*(1+\om_*^2)^{-\frac12}[v_1\mathtt v_{1} + v_{-1}\mathtt v_{-1}], \label{projection.onto.zn.I}
\\
&|\Pi_{Z_N}J_0\pa_\th v|^2
= 4\om_*^2\ell_*^2(1+\om_*^2)^{-1}[|v_1|^2 + |v_{-1}|^2].
\label{approx.I.Z.norm}
\end{align}

\end{lemma}

\begin{proof}
$(i)$ It follows by Lemma \ref{Range.eq.solv.}.

$(ii)$ Recall that $w(\om,0) = 0$, see \eqref{w.values}. 
Hence $g(\om,0) = \grad \mI(0) = 0$ by Lemma \ref{lemma:rotation} and also the second identity in \eqref{der.g}.
One has $f(\om,0) = 0$ because $\mF(\om,0) = 0$ for all $\om$.  
Also, 
\[
df(\om,v)\tilde v 
= \Pi_{Z_N} (d\mF)(\om, v + w(\om,v)) [ \tilde v + \pa_v w(\om,v)\tilde v ].
\]
At $v=0$, this gives 
\begin{equation} \label{pav.f.om.0}
df(\om, 0)\tilde v 
= \Pi_{Z_N} d\mF(\om, 0) [ \tilde v + \pa_v w(\om, 0)\tilde v ].
\end{equation}
This implies the second identity in \eqref{der.f} because, at $\om = \om_*$, 
$\pa_u\mF(\om_*,0)$ is the operator $\mL_{\om_*}$,  
whose range is orthogonal to $Z$. 
From \eqref{pav.f.om.0} we also obtain 
\[
\pa_{\om}\{df(\om ,0)\tilde v\} 
= \Pi_{Z_N} \pa_{\om}\{d\mF(\om, 0) [ \tilde v + \pa_v w(\om, 0)\tilde v ]\}
+ \Pi_{Z_N} d\mF(\om, 0) [ \pa_{\om}\{dw(\om, 0)\tilde v\} ],
\]
and, at $\om = \om_*$, 
\[
\pa_{\om v} f(\om_*,0) 
= \Pi_{Z_N} \pa_{\om u} \mF(\om_* , 0) 
\]
because $\pa_v w(\om_*, 0) = 0$, see \eqref{w.values},
and because the range of $\pa_u\mF(\om_*, 0) = \mL_{\om_*}$ is orthogonal to $Z$.   
By the definition \eqref{def.mF} of $\mF$, 
one has $\pa_{\om u} \mF(\om,u) = - d (\grad \mI)(u)$, 
and, as already observed in the last identity of \eqref{der.g}, 
$d(\grad \mI)(0) = J_0 \pa_\th$. This proves the third identity in \eqref{der.f}. 

The identities for $F$ and its derivatives follow from \eqref{der.f}, \eqref{der.g} 
and the product rule; for the last identity, we also use the basic fact that 
$\la \Pi_{Z_N} a, a \ra = \la \Pi_{Z_N} a, \Pi_{Z_N} a \ra = | \Pi_{Z_N} a |^2$.

$(iii)$ Direct computations.
\end{proof}

Now, we are in position to make a suitable choice of $\om=\om(v)$:

\begin{lemma}
\label{lemma:choice.of.omega} 
In the assumptions of Lemma \ref{lemma:technicalities}, there exist $\e_1 \in (0, \e_0]$, $b_1, C > 0$ 
and a function $\om : B_{V_N}(\e_1) \longmapsto \R$, $v \longmapsto \om(v)$, 
where $B_{V_N}(\e_1) := \{ v \in V_N : |v| < \e_1 \}$, 
which is Lipschitz continuous in $B_{V_N}(\e_1)$, 
analytic in $B_{V_N}(\e_1) \setminus \{ 0 \}$, 
such that $\om(0) = \om_*$,
\begin{equation}  \label{IFT.om}
F(\om(v), v) = 0
\end{equation}
for all $v \in B_{V_N}(\e_1)$, 
and, if $(\om,v)$ satisfies $F(\om,v) = 0$ 
with $|\om - \om_*| < b_1$, $v \in B_{V_N}(\e_1)$, 
then $\om = \om(v)$. 
Moreover, the graph $\{ (\om(v) , v) : v \in B_{V_N}(\e_1) \} \subset \R \times V_N$ 
is contained in the open set $\mU_{\e_0} \subset \R \times V$ 
where the function $w$ constructed in Lemma \ref{Range.eq.solv.} is defined, and 
\begin{equation} \label{om.estimate}
|\om(v)-\om_*| = O(|v|)
\quad \forall v \in B_{V_N}(\e_1).
\end{equation}
Also, for all $\alpha \in \T^1$, one has 
\begin{equation} \label{om.group.action}
\om(\mT_\alpha v) = \om(v).
\end{equation}
\end{lemma}

\begin{proof} 
By \eqref{der.F}, the expansion of the analytic function $v \longmapsto F(\om,v)$ is 
\begin{equation} \label{F.Taylor.in.v}
F(\om,v) = \frac12 d^2 F(\om, 0)[v, v] + O(|v|^3).
\end{equation} 
Consider $v \in V_N$, 
and introduce polar coordinates 
$\rho = |v|$, $y = v / |v|$ on $V_N \setminus \{ 0 \}$. 
The function  
\[
\Phi(\om, \rho, y) := \rho^{-2} F(\om, \rho y)
\]
is defined for $|\om - \om_*| < \e_0$, 
$\rho \in (0, \e_0)$, $y$ in the unit sphere $\{ |y| = 1 \}$ of $V_N$.
In fact, replacing $v$ with $\rho y$ in the converging Taylor series of $F(\om,v)$ 
centered at $(\om_*, 0)$, 
one obtains that $\Phi(\om, \rho, y)$ is a well-defined, converging power series 
in the open set $\mD := \{ (\om, \rho, y) : 
|\om - \om_*| < \e_0$, $\rho \in (0, \frac12 \e_0)$, $|y| < 2\}$. 
By \eqref{F.Taylor.in.v}, $\Phi(\om, \rho, y)$ converges to $\frac12 \pa_{vv} F(\om,0)[y,y]$ 
as $\rho \to 0$, and therefore $\Phi$ has a removable singularity at $\rho = 0$. 
Hence $\Phi$ has an analytic extension to the open set 
$\mD_1 := \{ (\om, \rho, y) : |\om - \om_*| < \e_0$, $|\rho| < \frac12 \e_0$, $|y| < 2\}$. 
We also denote this extension by $\Phi$. 
By \eqref{der.F} and \eqref{approx.I.Z.norm}, one has 
\begin{align*}
\Phi(\om_*, 0, y_0) 
& = \tfrac12 dF(\om_*, 0)[y_0, y_0] = 0, 
\\ 
\pa_\om \Phi (\om_*, 0, y_0) 
& = \tfrac12 \pa_\om\{d^2 F(\om_*, 0)[y_0, y_0]\} 
= - \tfrac12 |\Pi_{Z_N} J_0 \pa_\th y_0|^2
\neq 0
\end{align*}
for any $|y_0| = 1$. 
Hence, by the implicit function theorem for analytic functions, 
there exists a function $\Om(\rho, y)$ such that $\Om(0, y_0) = \om_*$ and 
\[
\Phi( \Om(\rho, y), \rho, y) = 0
\]
for all $(\rho,y)$ with $|\rho| < \e_1$, $|y - y_0| < \e_1$, 
for some $\e_1 \in (0, \e_0]$. Moreover, $\e_1$ is independent of $y_0$  
because the unit sphere of $V_N$ is compact. 
Finally, given $v \in V_N$, $|v| < \e_1$, we define 
$\om(v) := \Om(\rho, y)$ with $\rho = |v|$ and $y = v/|v|$ for $v \neq 0$,
and $\om(0) := \om_*$. 
The Lipschitz estimate \eqref{om.estimate} holds because 
$|\om(v) - \om_*| = |\Om(\rho,y) - \Om(0,y)| \leq C \rho$. 

By \eqref{function.for.om(v)}, 
the first identity in \eqref{torus.inv.lyap.schmidt}, 
the first one in \eqref{F.equiv.},
and \eqref{w.equiv.},
one has that for any $\alpha\in\T^1$,
\[
\mT_\alpha f(\om, v) = f(\om, \mT_\alpha v), \quad \ 
\mT_\alpha g(\om, v) = g(\om, \mT_\alpha v).
\]
Hence, by \eqref{F.def.}, one has 
\begin{equation} \label{F.group.action}
F(\om, \mT_\alpha v) = F(\om, v)
\end{equation}
because $\mT_\alpha^*  = \mT_\alpha^{-1}$.
By \eqref{IFT.om} with $\mT_\alpha v$ in place of $v$, one has 
$F(\om(\mT_\alpha v), \mT_\alpha v) = 0$. On the other hand, 
by \eqref{F.group.action}, one also has that 
$F(\om(\mT_\alpha v), \mT_\alpha v) = F(\om(\mT_\alpha v), v)$. 
Therefore 
\[
F(\om(\mT_\alpha v), v) = 0,
\]
and, by the uniqueness property of the implicit function, 
we obtain \eqref{om.group.action}. 
\end{proof}

Now, given $B_{V_N}(\e_1)$ as in Lemma \ref{lemma:choice.of.omega}, let us define the following objects for all $v\in B_{V_N}(\e_1)$ and $a\in\R$:
\begin{align}&u(v):=v+w(\om(v);v), \label{u.v.function}
\\&\mH_N(v):=\mH(u(v)),\qquad\mI_N(v):=\mI(u(v)),\qquad\mE_N(u(v)):=\mE(u(v)), \label{restricted.functionals}
\\&\mS_N(a):=\{v\in B_{V_N}(\e_1)\colon\,\mI_N(v)=a\}.  \label{restricted.constraint}
\end{align}
These quantities are nothing less but those already introduced before but restricted to $V_N$; the new one is only the constraint $\mS_N(a)$ of the conservation of the angular momentum restricted on $V_N$. The following lemma describes the regularity of \eqref{u.v.function} and the functionals in \eqref{restricted.functionals}:

\begin{lemma}\label{lemma:regularity}
Given the notations and the assumptions made above, the following facts hold.

$(i)$ The function $B_{V_N}(\e_1) \to W^{\mathfrak{s},s}$, $v \longmapsto u(v)$ 
is analytic in $B_{V_N}(\e_1) \setminus \{ 0 \}$,
it is of class $C^1 ( B_{V_N}(\e_1) )$, 
and its differential is Lipschitz continuous in $B_{V_N}(\e_1)$. 
Moreover $w(\om(\mT_\alpha v), \mT_\alpha v) = \mT_\alpha w(\om(v), v)$ for all $\alpha \in \T^1$. Same holds for $u(v)$.

$(ii)$ The functionals $\mH_N : B_{V_N}(\e_1) \longmapsto \R$ 
and $\mI_N : B_{V_N}(\e_1) \longmapsto \R$ 
are analytic in $B_{V_N}(\e_1) \setminus \{ 0 \}$
and of class $C^2(B_{V_N}(\e_1))$, 
and their second order differentials are Lipschitz continuous in $B_{V_N}(\e_1)$.
Moreover $\mH_N \circ \mT_\alpha = \mH_N$ 
and $\mI_N \circ \mT_\alpha = \mI_N$ for all $\alpha \in \T^1$. 
    
\end{lemma}

\begin{proof}
$(i)$ The function $v \longmapsto w(\om(v),v)$ is Lipschitz continuous in $B_{V_N}(\e_1)$
and analytic in $B_{V_N}(\e_1) \setminus \{ 0 \}$ 
by Lemmas \ref{Range.eq.solv.} and \ref{lemma:choice.of.omega}.
By \eqref{w.est.} and \eqref{om.estimate}, one has 
\begin{equation} \label{w.comp.estimate}
\| w(\om(v), v) \| = O(|v|^2)
\end{equation}
for all $v \in B_{V_N}(\e_1)$. 
Hence the function $v \longmapsto w(\om(v),v)$ is differentiable at $v=0$ 
with zero differential. 
Its differential at any point $v \in B_{V_N}(\e_1) \setminus \{ 0 \}$ 
in direction $\tilde v \in V_N$ is 
\begin{equation}  \label{der.w.comp}
d\{ w( \om(v) , v) \}\tilde v 
= d w( \om(v) , v)[d\om(v)\tilde v] 
+ dw( \om(v), v)\tilde v.
\end{equation}
For $v \to 0$, one has 
$(\om(v), v) \to (\om_*,0)$ because the function $\om(v)$ in Lemma \ref{lemma:choice.of.omega}
is continuous. Moreover,
$(\pa_v w)(\om(v), v) \to \pa_v w(\om_*, 0) = 0$ 
and $(\pa_\om w)(\om(v), v) \to \pa_\om w(\om_*, 0) = 0$,  
because the function $w$ in Lemma \ref{Range.eq.solv.} is analytic. 
By \eqref{w.values}, $d\om(v)\tilde v$, which is defined for $v \neq 0$, remains bounded as $v \to 0$
because $\om(v)$ is Lipschitz. 
Hence $\pa_v \{ w( \om(v) , v) \} \to 0$, which implies that $w(\om(v), v)$ 
is of class $C^1 ( B_{V_N}(\e_1) )$. 
Moreover, 
$|d\om(v)\tilde v| \leq C |\tilde v|$ for $0 < |v| < \e_1$ 
because $\om(v)$ is Lipschitz, 
and 
\[
\| (\pa_\om w)( \om(v) , v) \| \leq C |v|, \quad \  
\| d w( \om(v), v)\tilde v \| \leq C |v|\cdot|\tilde v|,
\]
because the function $w$ in Lemma \ref{Range.eq.solv.} is analytic 
and by \eqref{om.estimate}.
Hence, by \eqref{der.w.comp}, 
\[
\| d\{ w( \om(v) , v) \}\tilde v \|
\leq C |v|\cdot|\tilde v|
\]
for $v \in B_{V_N}(\e_1) \setminus \{ 0 \}$, 
so that the map $v \longmapsto \pa_v \{ w( \om(v) , v) \}$ is Lipschitz continuous around $v=0$. 

The fact that $w(\om(\mT_\alpha v), \mT_\alpha v) = \mT_\alpha w(\om(v), v)$ follows from \eqref{w.equiv.} and \eqref{om.group.action}. Finally, the same regularity properties hold for $u(v)$ because of \eqref{u.v.function}.

$(ii)$ From Lemma
\ref{ww.rotating.wave.eq}, point $(i)$ and Lemma \ref{lemma:rotation}, equations \eqref{H.trans.inv} and \eqref{I.equiv}
we deduce the analyticity, the $C^1$ regularity with Lipschitz differentials 
and the invariance with respect to the group action $\mT_\alpha$ for both $\mH_N, \mI_N$. Now, we want to prove the higher regularity: we will do it for $\mI_N$, since for $\mH_N$ the computations are very similar.
By \eqref{restricted.functionals}, 
the differential of $\mI_{N}$ at a point $v \in B_{V_N}(\e_1)$ in direction $\tilde v \in V_N$ is 
\[
d\mI_{N}(v)\tilde v 
= d\mI(u(v))[ du(v)\tilde v], 
\quad \ 
du(v)\tilde v = \tilde v + d\{ w(\om(v),v) \}\tilde v.
\]
The map $v \longmapsto d\mI_{N}(v)$ is Lipschitz. 
At $v \neq 0$, its differential in direction $\tilde z \in V_N$ is 
\[
d^2\mI_{N}(v)[\tilde z, \tilde v] 
= d^2\mI(u(v))[ du(v)\tilde z, du(v)\tilde v ] 
+ d\mI(u(v))[ d^2 u(v)[ \tilde z, \tilde v] ].
\]
As $v \to 0$, one has $u(v) \to 0$, $du(v)\tilde v \to \tilde v$, 
$d\mI(u(v)) \to d\mI(0) = 0$, and $d^2\mI(u(v)) \to d^2\mI(0)$, 
while the bilinear map $d^2u(v) = d^2\{ w( \om(v), v ) \}$ 
remains bounded, i.e., $\| du(v)[ \tilde z, \tilde v] \| \leq C |\tilde z|\cdot|\tilde v|$
uniformly as $v \to 0$, because the map $v \longmapsto du(v)$ is Lipschitz. 
Hence, $d^2\mI_{N}(v)[\tilde z, \tilde v]$ converges to 
$d^2\mI(0)[\tilde z, \tilde v]$ as $v \to 0$. 
Similarly, one proves that 
\[
| d\mI_{N}(v)\tilde v - d^2\mI(0)[v, \tilde v] | 
\leq C |v|^2\cdot|\tilde v|
\]
for $v \neq 0$. This implies that the map $v \longmapsto d\mI_{N}(v)$ is differentiable at $v=0$, 
with differential $d^2\mI_{N}(0) = d^2\mI(0)$. 
Also, from the limit already proved it follows that $\mI_{N}$ is of class $C^2$. 
The Lipschitz estimate for the second order differential
\[
| d^2\mI_{N}(v)[ \tilde z, \tilde v] - d^2\mI_{N}(0)[ \tilde z, \tilde v] | 
\leq C |v|\cdot|\tilde z|\cdot|\tilde v|
\]
is proved similarly.
    
\end{proof}

The following lemma describes the topological properties of the constraint $\mS_{V_N}(a)$, which can be obtained thanks to an asymptotic description of $\mI_N$ close to $v=0$:
\begin{lemma}\label{constraint.top.desc.}
In the above notations and assumptions, one has the following facts.

$(i)$ The functional $\mI_{N} : B_{V_N}(\e_1) \longmapsto \R$ defined in \eqref{restricted.functionals} 
satisfies
\begin{equation}\label{IZ.approx}
\mathcal{I}_{N} (v) = \mathcal{I}_0 (v) + R(v), 
\quad \ |R(v)| = O(|v|^3),
\end{equation}
where 
\begin{equation} \label{def.mI.0}
\mathcal{I}_0(v) := 
\frac12 \la J_0 \pa_\th v , v \ra_{(L^2)^2} 
= \om_*\ell_*(1+\om_*^2)^{-1}[|v_1|^2 + |v_{-1}|^2],
\end{equation}
for $v = v_1\, \mathtt v_{1} + v_{-1}\mathtt v_{-1}$, 
where $J_0$ is defined in \eqref{J0}.

$(ii)$ There exist a constant $\e_2 \in (0, \e_1]$ 
and a map $\psi : B_{V_N}(\e_2) \longmapsto B_{V_N}(\e_1)$ such that 
\begin{equation} \label{mI.VN.psi}
\mI_{N}(\psi(v)) = |v|^2 
\end{equation}
for all $v \in B_{V_N}(\e_2)$. 
The map $\psi$ is analytic in $B_{V_N}(\e_2) \setminus \{ 0 \}$, 
it is of class $C^1( B_{V_N}(\e_2) )$, with Lipschitz differential,  
it is a diffeomorphism of $B_{V_N}(\e_2)$ onto its image $\psi( B_{V_N}(\e_2) )$, 
which is an open neighborhood of $\psi(0)=0$ contained in $B_{V_N}(\e_1)$,   
and 
\begin{equation} \label{psi.group.action}
\psi \circ \mT_\alpha = \mT_\alpha \circ \psi
\end{equation}
for all $\alpha \in \T$. 
Moreover, there exists $a_0 > 0$ such that, for all $a \in (0, a_0)$, 
the set $\mS_{N}(a)$ in \eqref{restricted.constraint} is 
\[
\mS_{N}(a) = \psi \big( \mathtt{S}(a) \big) 
\]
where
\begin{equation} \label{def.mS.0.VN.a}
\mathtt{S}(a) := \{ v \in V_N : |v|^2 = a \},
\end{equation}
and thus $\mS_{N}(a)$ is an analytic connected compact manifold of dimension $1$ embedded in $V_N$.
Its tangent and normal space at a point $v \in \mS_{N}(a)$ are
\begin{align} 
T_v( \mS_{N}(a) ) 
& = \{ \tilde v \in V_N : \la \grad \mI_{N} (v) , \tilde v \ra_{(L^2)^2} = 0 \}, 
\label{tangent.space}
\\
N_v( \mS_{N}(a) ) 
& = \{ \lm \grad \mI_{N} (v) : \lm \in \R \},
\label{normal.space}
\end{align}
where $\grad \mI_{N}(v) \neq 0$ for all $v \in \mS_{N}(a)$, all $a \in (0, a_0)$.

\end{lemma}

\begin{proof}
$(i)$ The Taylor expansion around $u=0$ 
of the analytic functional $\mI$ defined in \eqref{def.mI} is 
\begin{equation} \label{Taylor.mI}
\mI(u) = \mI_0(u) + O(|u|^3),
\end{equation} 
with $\mI_0$ defined in \eqref{def.mI.0}, 
because $\mI(0) = 0$, $\grad \mI(0) = 0$, and 
\[ 
\frac12 d^2\mI(0)[u,u] 
= \frac12 \la d(\grad \mI)(0)u, u \ra_{(L^2)^2} 
= \frac12 \la J_0 \pa_\th u , u \ra_{(L^2)^2} 
= \mI_0(u).
\]
Note that the identity $d(\grad \mI)(0) = J_0 \pa_\th$ was already observed in \eqref{der.g}.  
Recalling \eqref{w.comp.estimate}, 
plugging $u(v) = v + w(\om(v),v) = v + O(|v|^2)$ 
into the expansion \eqref{Taylor.mI} gives the first identity in \eqref{IZ.approx}. 
The identity in \eqref{def.mI.0} follows from 
the formula of $J_0 \pa_\th v$ in \eqref{formula.J0.mM.v}
and the definition \eqref{nondeg.space} of $\mathtt v_{1},\mathtt v_{-1}$. 

$(ii)$ Define the diagonal linear map 
\begin{equation} \label{def.Lm}
\Lm : V_N \longmapsto V_N, \quad 
\Lm \mathtt{v}_{m} :=\lambda= \lambda_m\mathtt{v}_{m}, \quad 
\lm_{m} := (1 + \om_*^2)^{\frac12} (\om_*\ell_*)^{-\frac12},\quad m\in\{-1,1\}.
\end{equation}  
By \eqref{IZ.approx}, one has 
\begin{equation} \label{hat.square}
\mI_0(\Lm v) = |v_{1}|^2 + |v_{-1}|^2 = |v|^2.
\end{equation}
Given $v$, we look for a real number $\mu$ such that 
\begin{equation} \label{19:22}
\mI_{N}( (1+\mu) \Lambda v ) = |v|^2. 
\end{equation}
We look for $\mu$ in the interval $[-\delta, \delta]$, 
with $\delta = \frac14$, 
and we assume that $|v| < \e_2$, where $\e_2 \in (0, \e_1]$ is such that 
$\frac54 |\Lm| \e_2 < \e_1$, so that, for $v \in B_{V_N}(\e_2)$, 
the point $(1+\mu) \Lm v$ is in the ball $B_{V_N}(\e_1)$ where $\mI_{N}$ is defined. 
Let $R$ be the remainder in \eqref{IZ.approx}. 
By \eqref{hat.square}, one has 
$\mI_0( (1+\mu) \Lm v ) = (1+\mu)^2 |v|^2$,  
and, since $\mI_{N}= \mI_0 + R$, 
\eqref{19:22} becomes 
\begin{equation} \label{19:23}
(2 \mu + \mu^2) |v|^2 + R( (1+\mu) \Lm v ) = 0.
\end{equation}
For $v \neq 0$, \eqref{19:23} is the fixed point equation $\mu = \mK(\mu)$
for the unknown $\mu$, where 
\begin{equation} \label{def.mK}
\mK(\mu) := - \frac{\mu^2}{2} - \frac{R( (1+\mu) \Lm v)}{2 |v|^2}.
\end{equation}
By the estimate in \eqref{IZ.approx}, 
for some constants $C_1, C_2$ one has
\[
| \mK(\mu) | \leq \frac{\delta^2}{2} + C_1 \e_2 \leq \delta, 
\quad \ 
| \mK'(\mu) | \leq \delta + C_2 \e_2 \leq \frac12 
\]
for all $|\mu| \leq \delta$, $|v| < \e_2$,  
provided $\e_2$ is sufficiently small, namely
$C_1 \e_2 \leq \frac12 \delta$ and $C_2 \e_2 \leq \frac14$. 
Hence, by the contraction mapping theorem, 
in the interval $[- \delta, \delta]$ there exists a unique fixed point of $\mK$, 
which we denote by $\mu(v)$. 
Hence, 
\begin{equation}  \label{def.psi.in.the.proof}
\mI_{N}( \psi (v) ) = |v|^2, \quad \ 
\psi(v) := (1 + \mu(v)) \Lm v,  
\end{equation}
for all $v \in B_{V_N}(\e_2) \setminus \{ 0 \}$. 
From the implicit function theorem 
applied to equation \eqref{19:23} around any pair $(v, \mu(v))$ 
it follows that the map $v \longmapsto \mu(v)$ is analytic in $B_{V_N}(\e_2) \setminus \{ 0 \}$. 
Moreover, 
$|\mu(v)| = |\mK(\mu(v))| \leq \frac12 \mu^2(v) + C_1 |v| \leq \frac18 |\mu(v)| + C_1 |v|$, 
whence 
\[
|\mu(v)| \leq C |v|.
\]
Thus, defining $\mu(0) := 0$, the function $\mu(v)$ is also Lipschitz in $B_{V_N}(\e_2)$. 
As a consequence, the function $\psi$ is analytic in $B_{V_N}(\e_2) \setminus \{ 0 \}$
and Lipschitz in $B_{V_N}(\e_2)$. In addition, $|\psi(v) - \psi(0) - \Lm v| \leq C |v|^2$, 
which means that $\psi$ is differentiable also at $v=0$, 
with differential $d\psi(0)\tilde v = \Lm \tilde v$. 
At $v \neq 0$, the differential is 
\[
d\psi(v)\tilde v = (d\mu(v)\tilde v) \Lm v + (1 + \mu(v)) \Lm \tilde v,
\]
and $d\psi(v) \to d\psi(0) = \Lm$ as $v \to 0$ because $\mu(v) \to 0$, $\Lm v \to 0$, 
and $|d\mu(v)\tilde v| \leq C |\tilde v|$ uniformly as $v \to 0$.
Thus, $\psi$ is of class $C^1$ in $B_{V_N}(\e_2)$. 
Moreover, 
\[
|d\psi(v)\tilde v - d\psi(0)\tilde v| \leq C |v|\cdot |\tilde v|,
\]
i.e., the differential map $v \longmapsto d\psi(v)$ is Lipschitz continuous. 

The function $\psi$ is a diffeomorphism of open sets of $V_N$, and, 
for each $a \in (0, a_0)$ sufficiently small, one has 
\begin{align*}
\mS_{N}(a) 
& = \{ v \in B_{V_N}(\e_1) : \mI_{N}(v) = a \} 
\\ 
& = \{ v = \psi(y) : y \in B_{V_N}(\e_2), \ a = \mI_{N}(v) = \mI_{N}(\psi(y)) = |y|^2 \}
\\ 
& = \psi ( \{ y \in V_N : |y|^2 = a \} ),
\end{align*}
namely $\mS_{V_N}(a)$ is the image of the sphere $\{ |y|^2 = a \}$ by the diffeomorphism $\psi$. 

By \eqref{torus.action.v.var}, we have
\begin{equation} \label{Lm.group.action} 
\Lm \mT_\alpha = \mT_\alpha \Lm.
\end{equation} 
and also
\begin{equation} \label{norm.group.action}
|\mT_\alpha v|^2 = |v|^2.
\end{equation}  
From \eqref{Lm.group.action}, \eqref{hat.square}, and \eqref{norm.group.action}, one has 
\[
\mI_0( \mT_\alpha \Lm v ) = \mI_0( \Lm \mT_\alpha v ) 
= |\mT_\alpha v|^2 = |v|^2 = \mI_0( \Lm v)
\] 
for all $v \in V_N$, 
and this implies that $\mI_0 \circ \mT_\alpha = \mI_0$ 
because $\{ \Lm v : v \in V_N \} = V_N$.
By Lemma \ref{lemma:regularity}, the functional $\mI_{V_N}$ has the same invariance property,  
and therefore the difference $R = \mI_{N} - \mI_0$ also satisfies
\begin{equation}  \label{mR.group.action}
R \circ \mT_\alpha = R
\end{equation}
for all $\alpha \in \T^1$. 
Now denote by $\mK(\mu, v)$ the scalar quantity $\mK(\mu)$ in \eqref{def.mK}. 
By uniqueness of fix point, we have proved that $\mK(\mu, v) = \mu$ if and only if $\mu=\mu(v)$. 
Moreover, by \eqref{norm.group.action} and \eqref{mR.group.action}, one has
$\mK(\mu, \mT_\alpha v) = \mK(\mu, v)$ for all pairs $(\mu, v)$, all $\alpha \in \T^1$. 
Then 
\[
\mu(v) = \mK( \mu(v), v) = \mK( \mu(v), \mT_\alpha v),
\]
whence $\mu(v) = \mu( \mT_\alpha v)$. 
This identity, together with \eqref{Lm.group.action}, 
gives \eqref{psi.group.action}.

\end{proof}

Now, we have all the tools to prove the solvability of the bifurcation equation \eqref{red.eq.}. Let us start showing the variational nature of the bifurcation equation \eqref{red.eq.} under the choice of $\om(v)$ done in Lemma \ref{lemma:choice.of.omega}:

\begin{lemma}
In the notations and assumptions made above, the following facts hold.

$(i)$ For all $a \in \R$, the functional $\mE_{N}\colon\, B_{V_N}(\e_1) \longmapsto \R$ defined in \eqref{restricted.functionals}
is Lipschitz continuous in $B_{V_N}(\e_1)$ and analytic in $B_{V_N}(\e_1) \setminus \{ 0 \}$.
Moreover, for all $\alpha \in \T^1$, one has  
\begin{equation} \label{mE.VN.group.action}
\mE_{N} \circ \mT_\alpha = \mE_{N}.
\end{equation}

$(ii)$ For any $v \in \mS_{N}(a)$, one has 
\begin{equation} \label{grad.mE.V}
\grad \mE_{N}(v) = \Pi_{Z_N} \mF (\om(v), u(v)).
\end{equation}

$(iii)$ If $v\in\mS_N(a)$ is a constrained critical point for $\mE_N$ along the constraint $\mS_N(a)$, then for all $a\in(0,a_0)$ with $\a_0$ as in Lemma \ref{constraint.top.desc.}, one has that also 
\begin{equation}\grad\mE_N(v)=0,
\end{equation}
and so v solves the bifurcation equation \eqref{red.eq.}.

$(iv)$ There exist at least two solutions for \eqref{red.eq.}.

\end{lemma}

\begin{proof}
$(i)$ Straightforward consequence of Lemma \ref{lemma:regularity}.

$(ii)$ The differential of the functional $\mE_{N}$ 
at a point $v \in B_{V_N}(\e_1) \setminus \{ 0 \}$ 
in any direction $\tilde v \in V_N$ is 
\begin{align*}
& d\mE_{N}(v)\tilde v 
= d\mathcal{H}_{N}(v)v - d\omega(v)[\tilde v] \big( \mI_{N} (v) - a \big) 
- \om(v) d\mI_{N}(v)\tilde v 
\notag \\
& \quad 
= \la (\grad \mH)(u(v)) , du(v)\tilde v \ra_{(L^2)^2} 
- d\omega(v)[\tilde v] \big( \mI_{N}(v) - a \big)  
- \om(v) \la (\grad \mI)(u(v)) , du(v)\tilde v \ra_{(L^2)^2} 
\notag \\
& \quad 
= \la \mF (\om(v), u(v)) , du(v)\tilde v \ra_{(L^2)^2}  
- d\omega(v)[\tilde v] \big( \mI_{N}(v) - a \big)  
\notag \\
& \quad 
= \la \mF (\om(v), u(v)) , \tilde v \ra_{(L^2)^2}  
+ \la \mF (\om(v), u(v)) , d\{ w(\om(v), v) \} \tilde v \ra_{(L^2)^2}  
- d\omega(v)[\tilde v] \big( \mI_{N} (v) - a \big)  
\notag \\
& \quad 
= \la \Pi_{Z_N} \mF (\om(v), u(v)) , \tilde v \ra_{(L^2)^2}  
+ \la \Pi_R \mF (\om(v), u(v)) , d\{ w(\om(v), v) \} \tilde v \ra_{(L^2)^2}  
\notag \\ 
& \quad \quad \ 
- d\omega(v)[\tilde v] \big( \mI_{N} (v) - a \big)  
\notag \\
& \quad 
= \la \Pi_{Z_N} \mF (\om(v), u(v)) , \tilde v \ra_{(L^2)^2}  
-d \omega(v)[\tilde v] \big( \mI_{N} (v) - a \big); 
\end{align*}
if we restrict to any $v\in\mS_N(a)$, then we get
\begin{equation*}d\mE_N(v)\tilde v=\la \Pi_{Z_N} \mF (\om(v), u(v)) , \tilde v \ra_{(L^2)^2},\end{equation*}
from which we have \eqref{grad.mE.V}.

$(iii)$ If $v\in\mS_N(a)$ is a constrained critical point for $\mE_N$ along the constraint $\mS_N(a)$, then there exists a constant $\lambda\in\R$ such that
\begin{equation*}\label{lagrange.multip.}\grad\mE_N(v)=\lambda\grad\mI_N(v);
\end{equation*}
Then, using \eqref{IFT.om}, \eqref{F.def.}, \eqref{function.for.om(v)}, 
the definition of $u(v)$ in \eqref{u.v.function},
\eqref{grad.mE.V}, and \eqref{lagrange.multip.}, 
we obtain
\begin{align} \label{0=lm.scal.prod}
0 & = F(\om(v), v) 
= \la \Pi_{Z_N} \mF (\om(v), u(v)) , (\grad \mI)(u(v)) \ra_{(L^2)^2}
= \lm \la \grad \mI_{N}(v) , (\grad \mI)(u(v)) \ra_{(L^2)^2}.
\end{align}
By \eqref{IZ.approx}, \eqref{Taylor.mI}, and \eqref{w.comp.estimate}, one has 
\[
\grad \mI_{N}(v) = \grad \mI_0(v) + O(|v|^2), 
\quad 
(\grad \mI)(u(v)) = \grad \mI_0(v) + O(|v|^2),
\]
and 
\[
\la \grad \mI_{N}(v) , (\grad \mI)(u(v)) \ra_{(L^2)^2} 
= |\grad \mI_0(v)|^2 + O(|v|^3) 
\geq \tfrac12 |\grad \mI_0(v)|^2 
> 0. 
\]
This means that the coefficient of $\lm$ in \eqref{0=lm.scal.prod} is nonzero, 
whence $\lm = 0$. Thus, by \eqref{lagrange.multip.}, $\grad \mE_{V_N}(v) = 0$, 
and since Lemma \ref{nondeg.space} holds, then \eqref{red.eq.} follows from \eqref{grad.mE.V}.

$(iv)$ It is enough to observe that by Lemma \ref{constraint.top.desc.}, the constraint $\mS_N(a)$ is a connected compact smooth manifold with no boundary, so any $C^1$ function over it attains at least one maximum and one minimum which must be critical points. 

\end{proof}

This would be enough to show that there are rotating waves. However, all the invariances and equivariances of functions, functionals and their gradients with respect to the torus action $\mT_\alpha$ provide us many more solutions, actually an orbit of solutions, which coincides with the constraint $\mS_N(a)$. We have indeed our final lemma:
\begin{lemma}
In the notations and assumptions above, the following facts hold.

$(i)$ Given an integer number $k\in\N$, let $A$ be an open $\T^1$-invariant neighborhood of the sphere $\S^{2k-1}$ and
let $\ph \in C^1(A, \R)$ be a $\T^1$-invariant function. 
Then there exist at least $k$ $\T^1$-orbits of critical points of $\ph$ 
restricted to $\S^{2k-1}$.

$(ii)$ For any $a \in (0, a_0)$, the functional $\mE_{N}$ has exactly one orbit $\{ \mT_\alpha v : \alpha \in \T^1 \}$, 
of constrained critical points on the constraint $\mS_{N}(a)$, which coincides with $\mS_N(a)$ itself. As a result, the restricted Hamiltonian $\mH_N$ is constant along $\mS_N(a)$.
    
\end{lemma}

\begin{proof}
$(i)$ This is Lemma $6.10$, \cite{Mahwin.Willem}, Section $6.4$.

$(ii)$ Fix $a \in (0, a_0)$. 
By Lemma \ref{constraint.top.desc.}, one has $\mS_{N}(a) = \psi (\mathtt S(a) )$. 
Consider the function
\begin{equation}  \label{def.comp.mE.VN.psi}
f : B_{V_N}(\e_2) \to \R, \quad \ f(y) := \mE_{N} (\psi(y)). 
\end{equation}
The differential of $f$ at a point $y \in B_{V_N}(\e_2)$
in direction $\tilde y \in V_N$ is 
\[
df(y)\tilde y = d\mE_{N}(\psi(y))[ d\psi(y)\tilde y],
\]
and, for $y \in \mathtt S(a)$, 
$\tilde y$ is in the tangent space $T_y( \mathtt S(a) )$ 
iff $d\psi(y)\tilde y$ is in the tangent space $T_{\psi(y)} ( \mS_{N}(a) )$. 
Hence a point $v = \psi(y) \in \mS_{N}(a)$ 
is a constrained critical point of $\mE_{N}$ 
on the constraint $\mS_{N}(a)$ iff 
$y \in \mathtt S(a)$ is a constrained critical point of $f$ 
on the constraint $\mathtt S(a)$. 
By \eqref{psi.group.action} and \eqref{mE.VN.group.action}, one has 
\begin{equation} \label{f.group.action}
f \circ \mT_\alpha = f
\end{equation}
for all $\alpha \in \T^1$. 
Then, by point $(i)$, 
$f$ has one orbit of critical points on $\mathtt S(a)$, then it coincides with $\mathtt S(a)$ itself.

\end{proof}

\begin{proof}[Proof of Theorem \ref{thm:main}] It is a consequence of what shown in Sections 3.2 and 3.3.
\end{proof}

\subsection{Existence of symmetric rotating waves}

In this section, we recover the result in \cite{Moon.Wu}, but in the setting we have built up to now.

Let us recall the reversibility operator $\mR$ defined in \eqref{revers.op.}, and let us recall the subspace of $L^2\times L^2$ of its fixed points:
\begin{align*}E:&=\{u=(\eta,\beta)\in(L^2)^2\colon\,\mR u=u\}
\\&=\{u=(\eta,\beta)\in(L^2)^2\colon\,(\eta(\th),\beta(\th))=(\eta(-\th),-\beta(-\th)),\,\forall\th\in\T^1\},
\end{align*}
which is made of the $u$ such that $\eta$ is even and $\beta$ is odd.
In terms of Fourier expansion, $E$ can be characterized as follows:
\begin{equation}E=\{u=(\eta,\beta)\in(L^2)^2\colon\,\eta_{\ell,-1}=\beta_{\ell,-1}=0,\,\forall\ell\in\N,\,\beta_{0,0}=0\}.
\end{equation}
By \eqref{F.equiv.}, it follows that $E$ is an invariant subspace for $\mF$, that is,
\begin{equation}\label{E.F.inv.}\mF((\R\times E)\cap U)\subset\R\times E.
\end{equation}
Let us have a look to the linearized operator $L_{\om}$, see \eqref{lin.op.}. If we define $V_E:=V\cap E$ and $W_E:=W\cap E$, see \eqref{kernel.L.} and \eqref{W}, then
\begin{equation}\label{V.E}V_E=\{\lambda\mathtt v_{1}\colon\,\lambda\in\R\},\qquad W_E=V_E^{\perp_E}
\end{equation}
where $\mathtt v_1$ is defined in \eqref{nondeg.space}; one can also check that
\begin{align}&V_E\oplus W_E=E,\qquad\qquad R_E\oplus Z_E=E, \label{lyap.symm.dec}
\\&R_E=R\cap E=W_E,\qquad Z_E=V_E, \label{range.symm.dec}
\end{align}
and, calling $R_E^{\mathfrak{s},s}:=R_E\cap R^{\mathfrak{s},s}$ and $W_E^{\mathfrak{s},s}:=W_E\cap W^{\mathfrak{s},s}$ (see \eqref{lyap.schm.dec.} and \eqref{range.dec}), one can show that Lemmas \ref{L.invertibility} and \ref{Range.eq.solv.} still hold, apart from the torus action equivariance.

Then, we find an open set $\mU_{\e_0}\subset\R\times V_E$ where $w=w(\om;v)\in W_E^{\mathfrak{s},s}$, which satisfies \eqref{zeros.implicit.func.}-\eqref{w.est.}.

The existence of nonzero symmetric rotating waves would be then a consequence of the following transversality lemma:

\begin{lemma}\label{transv.lemma} 
The following facts hold.

$(i)$ In the notations and assumptions made above, it holds that
\begin{equation}\label{transv.}\pa_\th\mathtt v_1\notin R,
\end{equation}
where $R$ is defined in \eqref{R}.

$(ii)$(Crandall-Rabinowitz Bifurcation Theorem) Let $X,Y$ two Banach spaces, let $A\subseteq X$ be an open subset containing $0$, and let us suppose $F\colon\,\R\times A\longmapsto Y$ to be analytic and such that for all $\lambda\in\R$, $F(\lambda,0)=0$. Let $\lambda_*$ be such that $L:=\pa_u F(\lambda_*,0)$ is such that $V=\ker L$ is $1-$dimensional, $X=V\oplus W$ and $R=L(W)$ is closed and has codimension $1$. Let us suppose also that $\pa_{\lambda u}^2F(\lambda_*,0)[\mathtt u]\notin R$, where $\mathtt u$ is the generator of $V$.
Then, $\lambda_*$ is a bifurcation point for $F$, and the set of nontrivial solutions close to $(\lambda_*,0)$ is an unique analytique parametrized over $V$.

\end{lemma}

\begin{proof}
$(i)$ Since 
\begin{equation*}\mathtt v_1(\th)=\begin{pmatrix}\ph_{\ell_*,1}(\th) \\ -\om_*\ph_{\ell_*,-1}(\th)
\end{pmatrix} = \begin{pmatrix}\cos(\ell_*\th) \\ -\om_*\sin(\ell_*,\th)
\end{pmatrix},
\end{equation*}
then
\begin{equation*}\pa_\th\mathtt v_1=-\ell_*\begin{pmatrix}\sin(\ell_*\th) \\ \om_*\cos(\ell_*\th)
\end{pmatrix} = -\ell_*\begin{pmatrix}\ph_{\ell_*,-1} \\ \om_*\ph_{\ell_*,1}
\end{pmatrix}\in V=Z.
\end{equation*}
$(ii)$ See \cite{Ambrosetti.Prodi}, Theorem 4.1, Chapter 5.
    
\end{proof}

\begin{proof}[Proof of Theorem \ref{thm:bif}, point $(i)$]
It follows by Lemma \ref{transv.lemma}, where in point $(ii)$ we set $A:=E\cap U$, $X:=H^{\mathfrak{s},s+\frac32}\times H^{\mathfrak{s},s+1}$, $Y:=E\cap(H^{\mathfrak{s},s-\frac12}\times H^{\mathfrak{s},s})$, $F:=\mF$, $\lambda:=\om$, $\lambda_*:=\om_*$, $V:=V_E$, $W:=W_E$, $R:=R_E$ and $\mathtt u:=\mathtt v_1$.
\end{proof}

Now, we want to deal with the $c-$fold symmetry, with $c\in\N$ such that $c\ge2$. To this term, we recall
\begin{align*}L_c^2
&=\Big\{f\in L^2\colon\,f\Big(\th+\frac{2\pi}{c}\Big)=f(\th),\quad\forall\th\in\T^1\Big\}
\\&=\Big\{f\in L^2\colon\,f=\sum_{(\ell,m)\in\mT}f_{c\ell,m}\ph_{c\ell,m} \Big\}   \notag
\end{align*}
We notice that $L_c^2\times L_c^2$ is an invariant subspace for $\mF$. One can then restrict $\mF$ to $\R\times(U\cap(L_c^2\times L_c^2))$.

\begin{proof}[Proof of Theorem \ref{thm:bif}, points $(ii)$ and $(iii)$] 

Let us fix any $\ell_*\in\N$, and let us consider all the frequencies
\begin{equation}\label{freq.c.fold}\om_*:=\sqrt{\s_0}\cdot\sqrt{\frac{(c\ell_*)^2 - 1}{c\ell_*}}.
\end{equation}
Point $(ii)$ follows by repeating the same analysis done in Sections 4.2, 4.3; point $(iii)$ follows by repeating the computations done for proving point $(i)$.

\end{proof}

We conclude this section with the following remark. In Theorem \ref{thm:main}, we proved that for each (small enough) value $a$ of the angular momentum, we have a unique $\mT_\alpha-$orbit of rotating waves. In this orbit, one can also find the one of Theorem \ref{thm:bif}, point $(i)$, with the reversibility symmetry: as a result, we deduce that indeed, such an orbit is generated by the latter rotating wave, and so that is the only one that can be found up to translations of the argument. The same holds when it comes to the $c-$fold symmetry.

\begin{flushright}

\textbf{Giuseppe La Scala}

Mathematical and Physical Sciences for Advanced Materials and Technologies

Scuola Superiore Meridionale

Via Mezzocannone, 4, 80138 Naples, Italy

giuseppe.lascala-ssm@unina.it

\end{flushright}
\end{document}